\documentclass[preprint]{imsart}

\RequirePackage[OT1]{fontenc}
\RequirePackage{amsfonts,amsthm,amsmath}
\RequirePackage[numbers]{natbib}
\RequirePackage[colorlinks,citecolor=blue,urlcolor=blue]{hyperref}

\usepackage{esint}
\usepackage{mathrsfs}
\usepackage{bm}
\usepackage{enumerate}

\startlocaldefs
\numberwithin{equation}{section}
\theoremstyle{plain}
\newtheorem{thm}{Theorem}[section]
\endlocaldefs

\newtheorem{theorem}{Theorem}[section]
\newtheorem{lemma}[theorem]{Lemma}
\newtheorem{corollary}[theorem]{Corollary}

\theoremstyle{remark}
\newtheorem{remark}[theorem]{Remark}

\theoremstyle{definition}
\newtheorem{assumption}[theorem]{Assumption}

\newtheorem{definition}[theorem]{Definition}

\newcommand\cbrk{\text{$]$\kern-.15em$]$}}
\newcommand\opar{\text{\,\raise.2ex\hbox{${\scriptstyle|}$}\kern-.34em$($}}
\newcommand\cpar{\text{$)$\kern-.34em\raise.2ex\hbox{${\scriptstyle |}$}}\,}

\def\XXint#1#2#3{{\setbox0=\hbox{$#1{#2#3}{\int}$}
\vcenter{\hbox{$#2#3$}}\kern-.5\wd0}}

\newcommand\bL{\mathbb{L}}
\newcommand\bR{\mathbb{R}}

\newcommand\bH{\mathbb{H}}

\newcommand\bZ{\mathbb{Z}}

\newcommand\bD{\mathbb{D}}
\newcommand\bS{\mathbb{S}}

\newcommand\bE{\mathbb{E}}

\newcommand\cF{\mathcal{F}}

\newcommand\cM{\mathcal{M}}

\begin{document}
\begin{frontmatter}
\title\huge\textbf{ BOUNDARY BEHAVIOR AND INTERIOR H\"OLDER REGULARITY OF SOLUTION TO NONLINEAR STOCHASTIC PARTIAL DIFFERENTIAL EQUATIONS DRIVEN BY SPACE-TIME WHITE NOISE} 
\runtitle{Regularity of  SPDE driven by space-time white noise}

\begin{aug}
\author{\fnms{Beom-seok} \snm{Han}\thanksref{T1}\ead[label=e1]{hanbeom@korea.ac.kr}},
\author{\fnms{Kyeong-hun} \snm{Kim}\thanksref{T1}\ead[label=e2]{kyeonghun@korea.ac.kr}},

\thankstext{T1}{Supported by Samsung Science  and Technology Foundation under Project Number SSTF-BA1401-02.}

\runauthor{B. Han and K. Kim}

\affiliation{Department of mathematics, Korea university}

\address{Department of Mathematics, Korea University \\
1 anam-dong sungbuk-gu,  \\
Seoul, south Korea 136-701 \\
\printead{e1}\\
\printead{e2}}

\end{aug}
\begin{abstract}
We present uniqueness and existence  in weighted Sobolev spaces  of  the equation
\begin{equation*}
    \label{eqn main_0}
u_t=(au_{xx}+bu_x+cu)+ \xi |u|^{1+\lambda} {\dot{B}}, \quad\,\, t>0, \, x\in (0,1)
\end{equation*}
with  initial data $u(0,\cdot)=u_0$ and zero boundary data. Here $\lambda\in [0,1/2)$,  $\dot{B}$ is a space-time white noise, and the coefficients $a,b,c$ and the function $\xi$ depend on $(\omega,t,x)$ and the initial data $u_0$ depends on $(\omega,x)$. 

More importantly, we  obtain various interior H\"older regularities and boundary behaviors  of the solution. For instance,   if the initial data is in appropriate $L_p$ spaces, then for any  small $\varepsilon>0$ and $T<\infty$, almost surely
\begin{equation}
    \label{eqn 12.16}
    \rho^{-1/2-\kappa}u \in C^{\frac{1}{4}-\frac{\kappa}{2}-\varepsilon, \frac{1}{2}-\kappa-\varepsilon}_{t,x}([0,T]\times (0,1)), \quad \forall\, \kappa\in (\lambda, 1/2),  
\end{equation}
where $\rho(x)$ is the distance from $x$ to the boundary. Taking $\kappa \downarrow \lambda$, one gets the the maximal H\"older exponents in time and space, which  are $1/4-\lambda/2-\varepsilon$ and 
$1/2-\lambda-\varepsilon $ respectively. Also, letting $\kappa \uparrow 1/2$, one gets better decay or behavior near the boundary.
\end{abstract}

\begin{keyword}[class=MSC]
\kwd{60H15, 35R60.}
\end{keyword}

\begin{keyword}
\kwd{Nonlinear stochastic partial differential equations, Space-time white noise, Boundary behavior, Interior H\"older regularity.}
\end{keyword}

\end{frontmatter}

\section{Introduction}

In this article  we  present the unique solvability theory  in Sobolev spaces and interior H\"older regularities  of the solution to  the  zero-boundary value problem of   the equation
  \begin{equation} \label{main_equation}
du =  (au_{xx}+bu_x+cu)\, dt+ \xi |u|^{1+\lambda} dB_t, \,\, t>0, \, x\in I \,; \,u(0,\cdot) = u_0,
\end{equation}
where  $0\leq \lambda < \frac{1}{2}$, $I=(0,1)$ and $B_t$ is a standard cylindrical Brownian motion on $L_2(I)$.  The  initial data  $u_0=u_0(\omega,x)$ is nonnegative,
$\xi$  is  a bounded function depending on $(\omega,t,x)$, and 
the coefficients $a,b,c$ are measurable in $(\omega,t)$ and twice differentiable in $x$.

Here is a short description on the related works. 
 The equation of type
\begin{equation}
      \label{eqn 12.16.7}
      du=u_{xx}\,dt + |u|^{\gamma} \,dB_t, \quad t>0, x\in \bR
      \end{equation}
      has been considered as one of main problems in the theory of SPDEs and has been intensively studied.  
      See e.g. \cite{Kijung, kry1999analytic,mueller2009some,mueller1991long,mueller2014nonuniqueness,burdzy2010non, mytnik2011pathwise, Xiong}.
If $\gamma=1/2$ then equation \eqref{eqn 12.16.7} describes the super-process (or Dawson-Watanabe process), and strong uniqueness of non-negative solution was proved in \cite{Xiong}. If  $\gamma\in (0,3/4)$, then non-uniqueness can hold among signed solutions, and the uniqueness among non-negative functions is still open (see \cite{mueller2014nonuniqueness,burdzy2010non}). If  $\gamma\in (3/4, 1)$ then existence of weak solution and path-wise uniqueness was obtained in \cite{mytnik2011pathwise}, and if $\gamma\in [1,3/2)$ then long time existence was proved in \cite{mueller1991long} (on unit circle) and \cite{kry1999analytic} (on $\bR$). 

We remark that all of above   articles handle equation \eqref{eqn 12.16.7} on either $\bR$ or unit circle $S^1$, and  to the best of our knowledge,  general theory   handling equation \eqref{main_equation} in a domain with non-empty boundary is not available in the literature.

Comparing \eqref{main_equation} and \eqref{eqn 12.16.7}, note that $\gamma=1+\lambda$, and our equation is much general than \eqref{eqn 12.16.7} in the sense that the coefficients $a,b,c$ and the function $\xi$ are allowed to depend on $(\omega,t,x)$.   
Under the  assumptions described below \eqref{main_equation},  we prove that if a nonnegative function $u_0$ belongs to a Sobolev space $U^{1/2-\kappa}_{p,\theta}(I)$, $\kappa\in (\lambda, 1/2)$,  then the equation has a unique strong solution in $\mathfrak{H}^{1/2-\kappa}_{p,\theta,loc}(I,\infty)$. See Definition \ref{def_of_sol_1} for the notations. We only mention that $1/2-\kappa$ represents the number of differentiability, $p$ represents the summability, and $\theta$ is the parameter controlling the decay near the boundary.  

We also  obtain various interior H\"older regularities of solutions. For instance, (recall $\rho(x)=\text{dist}(x,\partial I)$) our result says that if $\kappa \in (\lambda, 1/2)$, $p>6(1-2\kappa)^{-1}$,   $\theta\in (0,p)$,  and 
\begin{equation}
    \label{eqn 12.16.3}
\rho^{\frac{1+\theta}{p}-1}u_0, \,\, \rho^{\frac{1+\theta}{p}} D_x u_{0}  \in L_p(\Omega\times I),
\end{equation}
 then  for any $\alpha, \beta$ satisfying 
$\frac{1}{p}<\alpha<\beta <(\frac{1}{4}-\frac{\kappa}{2}-\frac{1}{2p})$, 
\begin{equation}
     \label{eqn 12.16.2}
|\rho^{-1/2-\kappa+\frac{\theta-1}{p}}u|_{C^{\alpha-1/p}([0,T];C^{1/2-\kappa-2\beta-1/p}(I)}<\infty, \quad (a.s.). 
\end{equation}
Note that, due to Hardy's inequality,  \eqref{eqn 12.16.3}  holds for any $\theta>0$ if  $u_0 \in L_p(\Omega; \mathring{W}^1_p(I))$.
Thus if $u_0\in \cap_{p>1} L_p(\Omega; \mathring{W}^1_p(I))$  then   by  taking $\theta=1$ and $p$ sufficiently large in \eqref{eqn 12.16.2} we get  \eqref{eqn 12.16}, that is,
$$
  \rho^{-1/2-\kappa}u \in C^{\frac{1}{4}-\frac{\kappa}{2}-\varepsilon, \frac{1}{2}-\kappa-\varepsilon}_{t,x}([0,T]\times (0,1)), \quad \forall\, \kappa\in (\lambda, 1/2).
  $$
  This certainly implies (a.s.)
  $$
  \rho^{-1/2-\lambda}u\in C^{\frac{1}{4}-\frac{\lambda}{2}-\varepsilon, \frac{1}{2}-\lambda-\varepsilon}_{t,x}, \quad \forall \,  \text{small} \,\varepsilon >0,
  $$
  and for any $\varepsilon>0$,
  $$
  \sup_{t\leq T} |u(t,x)|\leq N(\omega) \rho^{1-\varepsilon}(x).
  $$
  In general, if \eqref{eqn 12.16.3} holds for a particular couple $(p,\theta)$ then one can get  the corresponding decay near the boundary from \eqref{eqn 12.16.2}.  Actually, our condition on $u_0  $ is much weaker than  \eqref{eqn 12.16.3}.

We finish the introduction with  the notation used in the article.
$\mathbb{N}$ and $\mathbb{Z}$ denote the natural number system
and the integer number system, respectively.
As usual $\bR$
stands for the set of all real numbers, and $\bR_+=\{x\in \bR: x>0\}$.  For $p \in [1,\infty)$, a normed space $F$,
and a  measure space $(X,\mathcal{M},\mu)$, $L_{p}(X,\cM,\mu;F)$
denotes the space of all $F$-valued $\mathcal{M}^{\mu}$-measurable functions
$u$ so that
\[
\left\Vert u\right\Vert _{L_{p}(X,\cM,\mu;F)}:=\left(\int_{X}\left\Vert u(x)\right\Vert _{F}^{p}\mu(dx)\right)^{1/p}<\infty,
\]
where $\mathcal{M}^{\mu}$ denotes the completion of $\cM$ with respect to the measure $\mu$. For $\alpha\in (0,1]$ and $T>0$, $C^{\alpha}([0,T];F)$ denotes the set of $F$-valued continuous functions $u$ such that
$$
|u|_{C^{\alpha}([0,T];F)}=\sup_{t\in [0,T]}|u(t)|_F+\sup_{s,t\in [0,T]} \frac{|u(t)-u(s)|_F}{|t-s|^{\alpha}}.
$$
We use  ``$:=$" to denote a definition. 
For $a,b\in \bR$, 
$a \wedge b := \min\{a,b\}$, $a \vee b := \max\{a,b\}$.
By $\cF$ and $\cF^{-1}$ we denote the one-dimensional Fourier transform and the inverse Fourier transform, respectively. That is,
$\cF(f)(\xi) := \int_{\bR} e^{-i x  \xi} f(x) dx$ and $\cF^{-1}(f)(x) := \frac{1}{2\pi}\int_{\bR} e^{ i\xi  x} f(\xi) d\xi$.

$N$ and $c$ denote  generic constants which can differ from line to line and
if we write $N=N(a,b,\ldots)$, then this means that the constant $N$ depends only on $a,b,\ldots$. Finally,
for functions  depending on $\omega$, $t$,  and $x$, the argument
$\omega \in \Omega$ will be usually omitted.

\section{Main result}

Let    $(\Omega,\cF,P)$ be a complete probability space,
$\{\cF_{t},t\geq0\}$  an increasing filtration of $\sigma$-fields
$\cF_{t}\subset\cF$ satisfying the usual conditions, and $B_t$
a standard cylindrical  Brownian motion on $L_2(I)$  relative to $\cF_t$.  

Let  $\{\eta_k : k=1,2,\cdots\}$   be an orthonormal basis on $L_2(I)$. Then  (see e.g. \cite{Dalang, kry1999analytic}) there exists a sequence of independent one-dimensional Wiener processes $\{w^k_t: k=1,2,\cdots \}$ such that $dB_t=\sum_k \eta_k dw^k_t$ and equation \eqref{main_equation} becomes
\begin{equation}
 \label{main_eqn_1}
du= (au_{xx}+bu_x+cu)\, dt+ \sum_{k=1}^{\infty} \xi |u|^{1+\lambda} \eta_k dw^k_t.
\end{equation}
Moreover, the series of stochastic integrals does not depend on the choice of $\{\eta_k\}$. Thus, without loss of generality, we may assume that $\eta_k$ is bounded for each $k$.

To present our results,  we introduce some  related function spaces and their properties. For $p>1$ and $\gamma \in \mathbb{R}$, let $H_p^\gamma(\mathbb{R})$ denote the class of all tempered distributions $u$ on $\mathbb{R}$ such that
$$ \| u \|_{H_p^\gamma} := \| (1-\Delta)^{\gamma/2} u\|_{L_p}<\infty,
$$
where 
$$ (1-\Delta)^{\gamma/2}u = \mathcal{F}^{-1}\left((1+|\xi|^2)^{\gamma/2}\mathcal{F}(u)\right).
$$
The space $H_p^\gamma(l_2) = H_p^\gamma(\mathbb{R};l_2)$ is defined similarly for $l_2$-valued functions $g=(g^1,g^2,\cdots)$, that is, 
$$ \|g\|_{H_{p}^\gamma(l_2)}:= \| | (1-\Delta)^{\gamma/2} g|_{l_2}\|_{L_p}.
$$
It is well known (see e.g. \cite[Section 8.3]{kry1999analytic}) that if $\kappa\in (0,1/2)$, then  we have
\begin{equation}
   \label{convolution}
   (1-\Delta)^{(-1/2-\kappa)/2}\, u=c(\kappa) \int_{\bR} R_{\kappa}(x-y)u(y)dy,
   \end{equation}
   where
\begin{equation} \label{kernel}
R_{\kappa}(x) :=  |x|^{-(1-2\kappa)/2}\int_0^\infty t^{-(5-2\kappa)/4} e^{-tx^2-\frac{1}{4t}} \,dt.
\end{equation}
 Observe that $R$  decays exponentially fast as $|x|\to\infty$, and behaves like $|x|^{-(1-2\kappa)/2}$ near $x=0$. Thus
 \begin{equation}
     \label{eqn 12.14.1}
 \|R_{\kappa}\|_{L_{2r}}<\infty \quad \text{if and only if }\quad r<\frac{1}{1-2\kappa}.
 \end{equation}

For $x\in I$, denote $\rho(x)=dist(x, \partial I)=\min\{x, 1-x\}$.  We  choose an infinitely differentiable function $\psi(x)$  which is comparable to $\rho(x)$ in $I$, that is, there is a constant $c>0$ such that 
\begin{equation}
     \label{eqn psi}
c^{-1}\psi(x)\leq  \rho(x)\leq c \psi(x), \quad \forall \, x\in I.
\end{equation}
Note that any smooth function $\psi$ satisfies \eqref{eqn psi} if $\psi(0)=\psi(1)=0$, $\psi>0$ in $I$, $\psi'(0)>0$, and $\psi'(1)<0$.
Next we fix
 a nonnegative function $\zeta\in C_c^\infty(\mathbb{R}_+)$ such that
\begin{equation}
   \label{eqn 11.5}
\sum_{n\in\mathbb{Z}}\zeta^p(e^n x)>c>0,\quad\forall x\in\mathbb{R}_+,
\end{equation}
where $c$ is a constant.  It is easy to check that any non-negative function $\zeta$ satisfies \eqref{eqn 11.5} if $\zeta>0$ on $[e^{n}, e^{n+1}]$ for some $n\in \bZ$.  For $x\in \bR$ and $n\in\mathbb{Z}$, define
$$ \zeta_n(x) =\begin{cases}  \zeta(e^n\psi(x)) &: \,x \in I\\
 0 &:\, x \not\in I.
 \end{cases}
$$
Note that $\zeta_n(x)=0$ if $x\in I$ and is sufficiently close to $\partial I$. Thus, 
$$
\zeta_n \in C^{\infty}_c(I), \quad \sup_{x\in \bR} |D^m_x\zeta_n(x)|\leq N e^{mn}.
$$
For any distribution $u$ on $I$, since $\zeta_n \in C^{\infty}_c(I)$, we can define $u\zeta_n$ as a distribution on $\bR$.  
For $\gamma, \theta \in \bR$ and $p>1$, let $H_{p,\theta}^\gamma(I)$ denote the set of all distributions $u$ on  $I$ such that
\begin{equation} \label{def}
\| u \|^p_{H_{p,\theta}^\gamma(I)} := \sum_{n\in\mathbb{Z}}e^{n\theta}\| \zeta_{-n}(e^n\cdot)u(e^n\cdot) \|^p_{H_p^{\gamma}}<\infty.
\end{equation}
Similarly, for $l_2$-valued functions $g=(g^1,g^2,\cdots)$, we define 
$$ \| g \|^p_{H_{p,\theta}^\gamma(I,l_2)} := \sum_{n\in\mathbb{Z}}e^{n\theta}\| \zeta_{-n}(e^n\cdot)g(e^n\cdot) \|^p_{H_p^{\gamma}(l_2)}.
$$
Denote
$$
L_{p,\theta}(I):=H^0_{p,\theta}(I), \quad L_{p,\theta}(I,l_2):=H^0_{p,\theta}(I,l_2).
$$
 The spaces $H^{\gamma}_{p,\theta}(I)$ are independent of the choice of $\psi$ and $\zeta$ (see \cite{kry1999weighted,lototsky2000sobolev}) and the norms introduced by other choices of $\psi$ and $\zeta$ are all equivalent. In particular, it is easy to check that if 
$\gamma$ is a non-negative integer then
\begin{equation} \label{weighted_prop_2}
\| u \|^p_{H_{p,\theta}^\gamma(I)}\sim\sum_{k \leq \gamma}\int_I |\rho^{k}(x)D^k u(x)|^p \rho^{\theta-1}(x)dx, 
\end{equation}
where $D^k u$ is the $k$-th derivative of $u$ with respect to $x$. Indeed, for instance if $\gamma=0$, then by \eqref{def} and  the change of variables $e^nx \to x$,
$$
\|u\|^p_{L_{p,\theta}(I)}=\int_{I}   \left[\sum_{n=-\infty}^{\infty} e^{n(\theta-1)}|\zeta(e^{-n}\psi(x))|^p\right] |u(x)|^p dx. $$
Thus for  \eqref{weighted_prop_2}  with $\gamma=0$,  we only use the fact (see \cite[Remark 1.3]{kry1999weighted} for detail) that for any 
$\eta\in C^{\infty}_c (\bR_+)$, 
$$
\sum_{n\in \bZ} e^{n(\theta-1)}|\eta(e^{-n}\psi(x))|^p\leq N  \psi^{\theta-1}(x),
$$
and the inverse inequality also holds if $\eta$ satisfies \eqref{eqn 11.5}. Similarly, one can  prove  \eqref{weighted_prop_2} with direct calculations for any non-negative integer.

\begin{remark}\label{indep_zeta}
 Since the choice of $\zeta$ does not affect the norm of $H_{p,\theta}^{\gamma}(I)$, considering $\zeta^2$ in place of $\zeta$, we can choose  $\zeta$ such that  both $\zeta$ and $\zeta^{1/2}$  
satisfy \eqref{eqn 11.5}.  Also, since $\zeta$  has compact support in $\bR_+$,  $\zeta_{-n}(x)=0$ for all  large $n$ (say for $n\geq n_0$) and thus  it is enough to  consider only  $n\leq n_0$ in the summation of \eqref{def}.  It  follows from \eqref{def} that if $\theta_1<\theta_2$ then
\begin{equation}
   \label{eqn 11.5.2}
\|u\|_{H^{\gamma}_{p,\theta_2}(I)}\leq N(\zeta, \theta_1, \theta_2,p) \|u\|_{H^{\gamma}_{p,\theta_1}(I)}.
\end{equation}

\end{remark}

For real-valued functions $a=a(x)$ and $l_2$-valued functions $b=(b^1(x),b^2(x),\cdots)$, and $n=0,1,2,\cdots$  denote
\begin{equation*}
\begin{gathered}
|a|_{C^n}=|a|_{n}=\sum_{0\leq k \leq n} \sup_{x\in I} |D^ka(x)|, \quad |b|_{C^n(l_2)}= |b|_{n,l_2}=\sum_{0\leq k\leq n} \sup_{x\in I} |D^k b(x)|_{l_2}.\end{gathered}
\end{equation*}
   Also we define interior H\"older norm
\begin{equation*}
\begin{gathered}
|a|^{(0)}_{n}=\sum_{0\leq k \leq n}  \sup_x |\psi^k  D^ka|, \quad |b|^{(0)}_{n,l_2}=\sum_{0\leq k \leq n} \sup_x |\psi^k  D^kb|_{l_2}.
\end{gathered}
\end{equation*}
Obviously, since $\psi^k$ is bounded,
$$
|a|^{(0)}_{n}\leq   N(\psi,n) |a|_n.
$$

Below we collect some other facts on the space $H_{p,\theta}^\gamma(I)$. See  \cite{kry1999weighted,  {kr99-1}, lototsky2000sobolev} for further properties of the weighted Sobolev space. For $\alpha\in\mathbb{R}$, we write $f\in\psi^{\alpha}H_{p,\theta}^\gamma(I)$ if $\psi^{-\alpha}f\in H_{p,\theta}^\gamma(I)$.

\begin{lemma}
\label{prop_of_bessel_space} Let $\theta,\gamma \in \bR$ and $p>1$. 
\begin{enumerate}[(i)]
\item The space  $C_c^\infty(I)$ is dense in $H_{p,\theta}^{\gamma}(I)$. 

\item (interior H\"oder estimate). Let $\gamma - d/p = m+\nu$ for some $m=0,1,\cdots$ and $\nu\in(0,1]$. Then for any  $i\in\{ 0,1,\cdots,m \}$, we have
\begin{equation*}
| \psi^{i+\theta/p}D^i u |_{C (I)} + | \psi^{m+\nu+\theta/p}D^m u |_{C^\nu (I)} \leq N \| u \|_{H_{p,\theta}^\gamma (I)}.
\end{equation*}
\item Let $\alpha\in\mathbb{R}$. Then $\psi^\alpha H_{p,\theta+\alpha p}^\gamma(I)=H_{p,\theta}^\gamma(I)$,
\begin{equation*}
\| u \|_{H_{p,\theta}^\gamma (I)} \leq N \| \psi^{-\alpha}u \|_{ H_{p,\theta+\alpha p}^\gamma (I) } \leq N \| u \|_{H_{p,\theta}^\gamma (I)}.
\end{equation*}

\item  The operators $D_x: H^{\gamma}_{p,\theta-p}(I) \to H^{\gamma-1}_{p,\theta}(I)$ and 
$D^2_x: H^{\gamma}_{p,\theta-p}(I) \to H^{\gamma-2}_{p,\theta+p}(I)$ are bounded. Moreover
$$
\|u_x\|_{H^{\gamma-1}_{p, \theta}(I)}+ \| u_{xx}\|_{H^{\gamma-2}_{p,\theta+p}(I)} \leq N \|u\|_{H^{\gamma}_{p,\theta-p}(I)}.
$$

\item (point-wise multiplier). If $n$ is an interger such that $n\geq|\gamma|$, we have
$$
\|a u\|_{H^{\gamma}_{p,\theta}(I)}\leq N(\gamma,n) |a|^{(0)}_{n} \|u\|_{H^{\gamma}_{p,\theta}(I)},
$$
$$
\|b u\|_{H^{\gamma}_{p,\theta}(I,l_2)}\leq N(\gamma,n) |b|^{(0)}_{n,l_2} \|u\|_{H^{\gamma}_{p,\theta}(I)}.
$$

\item (multiplicative inequality). Let 
\begin{equation*} \label{condition_of_constants_interpolation}
\begin{gathered}
\kappa\in[0,1],\quad p_i\in(1,\infty),\quad\gamma_i,\theta_i\in\mathbb{R},\quad i=0,1,\\
\gamma=\kappa\gamma_1+(1-\kappa)\gamma_0,\quad1/p=\kappa/p_1+(1-\kappa)/p_0,\\
\theta/p=\theta_1\kappa/p_1+\theta_0(1-\kappa)/p_0.
\end{gathered}
\end{equation*}
Then we have
\begin{equation*}
\|u\|_{H^\gamma_{p,\theta}(I)} \leq \|u\|^{\kappa}_{H^{\gamma_1}_{p_1,\theta_1}(I)}\|u\|^{1-\kappa}_{H^{\gamma_0}_{p_0,\theta_0}(I)}.
\end{equation*}

\item (duality). Let 
$$\gamma'=-\gamma, \quad 1/p+1/{p'}=1, \quad  \theta/p+\theta'/{p'}=1.$$
 Then $$(H^{\gamma}_{p,\theta}(I))^{*}=H^{\gamma'}_{p',\theta'}(I).
 $$

\item (complex interpolation). For $\kappa\in (0,1), p\in (1,\infty), \gamma_i, \theta_i \in \bR$, $i=0,1$,
$$
\theta=\kappa \theta_1+(1-\kappa)\theta_0, \quad \gamma=\kappa \gamma_1+(1-\kappa)\gamma_0,
$$
we have $[H^{\gamma_0}_{p,\theta_0}(I), \, H^{\gamma_1}_{p,\theta_1}(I)]_{\kappa}=H^{\gamma}_{p,\theta}(I)$.

\end{enumerate}
\end{lemma}
\begin{proof}
All the results above are proved by Krylov \cite{kry1999weighted, {kr99-1}} in the half line (or half space)  and extended to general domain by   Lototsky \cite{lototsky2000sobolev}.
In the half line, the space $H^{\gamma}_{p,\theta}(\bR_+)$ is defined by (formally) taking $\psi(x)=x$ so that $\zeta_{-n}(e^nx)=\zeta(x)$ and \eqref{def} becomes
$$
\| u \|^p_{H_{p,\theta}^\gamma(\bR_+)} := \sum_{n\in\mathbb{Z}}e^{n\theta}\| \zeta(x)u(e^n\cdot) \|^p_{H_p^{\gamma}}.
$$
We also remark that all the Krylov's proofs in \cite{kry1999weighted, {kr99-1}} work  in $I$ with almost no changes  if one replaces $\zeta(x)$ there by $\zeta_{-n}(e^nx)=\zeta(e^{-n}\psi (e^nx))$.
\end{proof}

\begin{corollary} \label{interpol_cor}
(i) Let $\gamma_1<\gamma_2<\gamma_3$ and $\theta_1>\theta_2>\theta_3$. Then for any $\varepsilon>0$, there exists $N = N(\gamma_i,\theta_i,\varepsilon)>0$ $(i = 1,2,3)$ such that
\begin{equation*}
\| u \|_{H_{p,\theta_2}^{\gamma_2}(I)} \leq \varepsilon \| u \|_{H_{p,\theta_3}^{\gamma_3}(I)} + N\| u \|_{H_{p,\theta_1}^{\gamma_1}(I)}.
\end{equation*}

(ii) Let $\gamma, \theta\in \bR$ and $1<p<q$. Then 
\begin{equation}
    \label{eqn 12.4.1}
H^{\gamma}_{q,\theta}(I) \subset H^{\gamma}_{p,\theta}(I).
\end{equation}
\end{corollary}

\begin{proof}
(i) Take $\kappa \in(0,1)$ such that 
$\theta_2 = \kappa \theta_3+(1-\kappa)\theta_1$, and put $\gamma_0 := \kappa \gamma_3+(1-\kappa)\gamma_1$. 
Then by Lemma \ref{prop_of_bessel_space} (vi), 
\begin{equation*}
\| u \|_{H_{p,\theta_2}^{\gamma_0}(I)} \leq \| u \|^{1-\kappa}_{H_{p,\theta_3}^{\gamma_3}(I)}\| u \|^{\kappa}_{H_{p,\theta_1}^{\gamma_1}(I)}.
\end{equation*}
Therefore, the claim follows if $\gamma_2\leq\gamma_0$. Now, assume $\gamma_2>\gamma_0$. Take $\kappa'\in(0,1)$ such that
\begin{equation*}
\gamma_2 = \kappa'\gamma_3+(1-\kappa')\gamma_1.
\end{equation*}
Since $\gamma_3>\gamma_1$ and $\kappa'\gamma_3+(1-\kappa')\gamma_1=\gamma_2>\gamma_0=\kappa\gamma_3+(1-\kappa)\gamma_1$, we have $\kappa'>\kappa$. Thus,
\begin{equation*}
\theta':=\kappa'\theta_3+(1-\kappa')\theta_1<\kappa\theta_3+(1-\kappa)\theta_1=\theta_2.
\end{equation*}
Therefore, by \eqref{eqn 11.5.2} and Lemma \ref{prop_of_bessel_space} (vi),
\begin{equation*}
\begin{aligned}
\| u \|_{H_{p,\theta_2}^{\gamma_2}(I)} \leq N \| u \|_{H_{p,\theta'}^{\gamma_2}(I)} \leq N \| u \|_{H_{p,\theta_3}^{\gamma_3}(I)}^{1-\kappa'}\| u \|_{H_{p,\theta_1}^{\gamma_1}(I)}^{\kappa'}.
\end{aligned}
\end{equation*}

(ii) Due to Lemma \ref{prop_of_bessel_space} (iii), for any $p'>1$,  $u\in H^{\gamma}_{p',\theta}(I)$ if and only if $\psi^{(\theta-1)/p}u\in H^{\gamma}_{p',1}(I)$. Therefore, considering 
$\psi^{(\theta-1)/p}u$ in place of $u$, we may assume that $\theta=1$.

 By \eqref{weighted_prop_2} and H\"older inequality, we  get \eqref{eqn 12.4.1} if $\gamma$ is any non-negative integer. The result for non-negative integers  and  complex interpolation (Lemma \ref{prop_of_bessel_space} (viii)) prove  \eqref{eqn 12.4.1} for any $\gamma\geq 0$.  The case $\gamma<0$  follows from  the duality.   Indeed,
 $$
 H^{\gamma}_{q,1}(I)=(H^{-\gamma}_{q',1}(I))^{*}\subset (H^{-\gamma}_{p',1}(I))^{*}= H^{\gamma}_{p,1}(I),
 $$
 where $1/p+1/{p'}=1, 1/q+1/{q'}=1$. For the second relation above we use the  result for $\gamma\geq 0$. The
corollary is proved.
\end{proof}

Now we define stochastic Banach spaces.  Let $\mathcal{P}$ denote the predictable $\sigma$-field related  to $\cF_t$.   For a stopping time $\tau$, define 
$$\opar0,\tau\cbrk:=\{ (\omega,t):0<t\leq\tau(\omega) \},
$$
$$ \mathbb{H}_{p,\theta}^{\gamma}(I,\tau) := L_p(\opar0,\tau\cbrk, \mathcal{P}, dP \times dt ; H_{p,\theta}^\gamma(I)),
$$
$$
\mathbb{H}_{p,\theta}^{\gamma}(I,\tau,l_2) := L_p(\opar0,\tau\cbrk,\mathcal{P}, dP \times dt;H_{p,\theta}^\gamma(I,l_2)),1
$$
$$
U_{p,\theta}^{\gamma}(I) :=  L_p(\Omega,\cF_0, dP ; H_{p,\theta+2-p}^{\gamma-2/p}(I)),
$$
where, for instance,
\begin{equation*}
\| u \|^p_{\bH^{\gamma}_{p,\theta}(I,\tau)} := \mathbb{E} \int^{\tau}_0 \| u(t) \|^p_{H^{\gamma}_{p,\theta}(I)}dt. 
\end{equation*}

\begin{definition} \label{def_of_sol_1}
For $u\in \mathbb{H}_{p,\theta-p}^{\gamma+2}(I,\tau)$, we write $u\in\mathfrak{H}_{p,\theta}^{\gamma+2}(I,\tau)$ if $u_0\in U_{p,\theta}^{\gamma+2}(I)$ and there exists $(f,g)\in
\mathbb{H}_{p,\theta+p}^{\gamma}(I,\tau)\times\mathbb{H}_{p,\theta}^{\gamma+1}(I,\tau,l_2)$ so that
\begin{equation*}
du = fdt+\sum_{k=1}^{\infty} g^k dw_t^k,\quad \,  t\in (0, \tau]\cap (0,\infty)\,; \quad u(0,\cdot) = u_0
\end{equation*}
in the sense of distributions, i.e., for any $\phi\in C_c^\infty(I)$, the equality
\begin{equation} \label{def_of_sol_2}
(u(t,\cdot),\phi) = (u_0,\phi) + \int_0^t(f(s,\cdot),\phi)ds + \sum_{k=1}^{\infty} \int_0^t(g^k(s,\cdot),\phi)dw_s^k
\end{equation}
holds for all $t\in [0,\tau]\cap [0,\infty)$ (a.s.).  In this case, we write
\begin{equation*}
\mathbb{D}u:= f,\quad \mathbb{S}u:=g.
\end{equation*}
The norm in $\mathfrak{H}_{p,\theta}^{\gamma+2}(I,\tau)$ is introduced by
\begin{equation*}
\begin{aligned}
\| u \|_{\mathfrak{H}_{p,\theta}^{\gamma+2}(I,\tau)} &:= \| u \|_{\mathbb{H}_{p,\theta-p}^{\gamma+2}(I,\tau)}+ \| \mathbb{D}u \|_{\mathbb{H}_{p,\theta+p}^{\gamma}(I,\tau)} \\
&\quad+ \| \mathbb{S}u \|_{\mathbb{H}_{p,\theta}^{\gamma+1}(I,\tau,l_2)} + \| u(0,\cdot) \|_{U_{p,\theta}^{\gamma+2}(I)}.
\end{aligned}
\end{equation*}
We write $u\in\mathfrak{H}_{p,\theta,loc}^{\gamma+2}(I,\tau)$ if there exists a sequence of bounded stopping times $\tau_n\uparrow \tau$ so that $u\in\mathfrak{H}_{p,\theta}^{\gamma+2}(I,\tau_n)$ for each $n$.  We say $u=v$ in $\mathfrak{H}_{p,\theta,loc}^{\gamma+2}(I,\tau)$ if there are bounded stopping times $\tau_n \uparrow \tau$ such that $u=v$ in $\mathfrak{H}_{p,\theta}^{\gamma+2}(I,\tau_n)$ for each $n$.
\end{definition}

\begin{remark}
\label{remark 12.4}
In Definition \ref{def_of_sol_1}, the spaces for $u$, $\bD u$ and $\bS u$ are defined in  such a way  that the operators
$$
D^2_x: u \to \bD u, \quad D_x: u\to \bS u
$$
are bounded.  To be more specific,   consider  the simple equation
$$
du=a u_{xx} \, dt+ \sum_{k=1}^{\infty} \sigma^k u_x \, dw^k_t=:\bD u dt +\bS^k u \,dw^k_t.
$$
By Lemma \ref{prop_of_bessel_space} (iii)-(v),   if $a,\sigma$ are sufficiently smooth and $u\in \mathbb{H}_{p,\theta-p}^{\gamma+2}(I,\tau)$, 
$$
\bD u= au_{xx} \in  \mathbb{H}_{p,\theta+p}^{\gamma}(I,\tau),\quad 
\bS u=\sigma u_x \in  \mathbb{H}_{p,\theta}^{\gamma+1}(I,\tau, l_2).
$$
\end{remark}

\begin{remark}
Let $p\geq 2$ and $\gamma, \theta\in \bR$. Then  for any $g\in \bH^{\gamma+1}_{p,\theta}(I,\tau,l_2)$, the series of stochastic integral in \eqref{def_of_sol_2} converges uniformly in $t$ in probability on $[0,\tau \wedge T]$ for any $T$, and consequently $(u(t,\cdot),\phi)$ is continuous in $t$. See e.g. \cite[Remark 3.2]{kry1999analytic} for details.
\end{remark}

\begin{remark}
     \label{remark 12.11.1}
    (i) By Lemma \ref{prop_of_bessel_space} (iii),  $\|u\|^p _{\mathfrak{H}_{p,\theta}^{\gamma+2}(I,\tau)} $ is equivalent to 
     \begin{align*}
       \|\psi^{-1}u\|^p_{\bH^{\gamma+2}_{p,\theta}(I,\tau)}+\|\psi \bD u\|^p_{\bH^{\gamma}_{p,\theta}(I,\tau)}
     &+\|\bS u\|^p_{\bH^{\gamma+1}_{p,\theta}(I,\tau,l_2)} \\
     &+\bE\|\psi^{2/p-1}u(0,\cdot)\|^p_{H^{\gamma+2-2/p}_{p,\theta}(I)}.
     \end{align*}
     It follows from \eqref{eqn 11.5.2}  and \eqref{eqn 12.4.1} that if $q \geq p$ and $\theta'\leq \theta$, then 
     \begin{equation}
         \label{eqn 12.11.5}
     \mathfrak{H}_{q,\theta'}^{\gamma+2}(I,\tau)  \subset  \mathfrak{H}_{q,\theta}^{\gamma+2}(I,\tau) \subset     \mathfrak{H}_{p,\theta}^{\gamma+2}(I,\tau).
     \end{equation}

(ii) Note that if $u_0 \in U^{\gamma+2}_{p,\theta}(I)$, then $ \bE \|\psi^{2\beta-1}u_0\|^p_{H^{\gamma+2-2\beta}_{p,\theta}(I)}<\infty$ if $\beta\geq 1/p$. Obviously, this is necessary for \eqref{Bessel_embedding}.
\end{remark}

\begin{thm} \label{embedding}

 (i) For any $p\geq 2$, $\gamma, \theta\in \bR$, $\mathfrak{H}_{p,\theta}^{\gamma+2}(I,\tau)$ is a Banach space.
 
 (ii)  Let $\tau\leq T$, $p>2$ and $1/2>\beta>\alpha>1/p$. Then for any  $u\in\mathfrak{H}_{p,\theta}^{\gamma+2}(I,\tau)$ we have $u\in C^{\alpha-1/p}([0,\tau];H_{p,\theta}^{\gamma+2-2\beta} (I))$ (a.s.) and
\begin{equation} 
                    \label{Bessel_embedding}
\mathbb{E}| \psi^{2\beta-1}u |^p_{C^{\alpha-1/p}([0,\tau];H_{p,\theta}^{\gamma+2-2\beta} (I))} \leq N(d,p,\alpha,\beta,T)\| u \|^p_{\mathfrak{H}_{p,\theta}^{\gamma+2}(I,\tau)}.
\end{equation}

(iii) If $p=2$, then \eqref{Bessel_embedding} holds with $\beta = 1/2$ and $\alpha = 1/p = 1/2$, i.e.,  $u\in C([0,\tau];H_{2,\theta}^{\gamma+1} (I))$ (a.s.), and
\begin{equation}
   \label{eqn 11.27.1}
 \mathbb{E}\sup_{t\leq \tau}\| u \|^2_{H_{2,\theta}^{\gamma+1} (I)}\leq N\| u \|_{\mathfrak{H}_{2,\theta}^{\gamma+2}(I,\tau)}^2.
\end{equation}
(iv) If $p\geq 2$ and $\tau\equiv T$, then for  any $t\leq T$,
\begin{equation}
    \label{eqn 11.27.3}
\|u\|^p_{\bH^{\gamma+1}_{p,\theta}(I,t)}\leq N(T,p,\theta,\gamma)\int^t_0 \|u\|^p_{\mathfrak{H}^{\gamma+2}_{p,\theta}(I,s)}\,ds.
\end{equation}
\begin{proof}
All the results are proved by  Krylov \cite{kry2001} on half spaces and then extended in e.g. \cite[Theorem 2.3]{kim2017regularity} to general $C^1$ domains. We only mention that  (iv) is an easy consequence of \eqref{eqn 11.5.2}, \eqref{Bessel_embedding}, and \eqref{eqn 11.27.1}.   Indeed,
\begin{eqnarray*}
\|u\|^p_{\bH^{\gamma+1}_{p,\theta}(I,t)}&=&\int^t_0 \bE\|u(s)\|^p_{H^{\gamma+1}_{p,\theta}(I)}ds\\
&\leq& \int^t_0 \bE \sup_{r\leq s}\|u(r)\|^p_{H^{\gamma+1}_{p,\theta}(I)}\,ds \leq N\int^t_0 \|u\|^p_{\mathfrak{H}^{\gamma+2}_{p,\theta}(I,s)}\,ds.
\end{eqnarray*}
 \end{proof}

\end{thm}

\begin{corollary} \label{embedding_corollary}
Let $\tau\leq T$, $\kappa\in (0,1/2)$, and 
\begin{equation} \label{condition_for_alpha_beta_2}
\frac{1}{p}<\alpha<\beta<  \frac{1}{4}-\frac{\kappa}{2}-\frac{1}{2p}.
\end{equation}
Then for any $\delta \in [0, \frac{1}{2}-\kappa-2\beta-\frac{1}{p})$
we have
\begin{equation} \label{embedding_remark}
\mathbb{E} |\psi^{-1/2-\kappa-1/p+\theta/p-\delta}u|^p_{C^{\alpha-1/p}([0,\tau];C^{1/2-\kappa-2\beta-1/p-\delta} (I))}\leq N\|u\|_{\mathfrak{H}^{1/2-\kappa}_{p,\theta} (I,\tau)}^p,
\end{equation}
where $N = N(d,p,\alpha,\beta, \kappa,T)$. 
\end{corollary}

\begin{proof}
Applying Lemma \ref{prop_of_bessel_space} (ii) for $\psi^{2\beta-1}u$ and $\gamma = 1/2-\kappa-2\beta-\delta$,  we get 
\begin{equation*}
\begin{aligned}
|\psi^{-1/2-\kappa-1/p+\theta/p-\delta}& u|_{C^{1/2-\kappa-2\beta-1/p-\delta}(I)}\\
&\leq N\| \psi^{2\beta-1}u \|_{H_{p,\theta}^{1/2-\kappa-2\beta-\delta}(I)} \\
&\leq N\| \psi^{2\beta-1}u \|_{H_{p,\theta}^{1/2-\kappa-2\beta}(I)}.
\end{aligned}
\end{equation*}
Hence  this and \eqref{Bessel_embedding}  prove the corollary. 
\end{proof}

Finally we  introduce our assumption and main results.

\begin{assumption} \label{ass coeff}
(i) The coefficients $a(t,x)$, $b(t,x)$, $c(t,x)$ are $\mathcal{P}\times\mathcal{B}(I)$-measurable functions depending on $(\omega,t,x)$. \\
(ii) There exist    constants  $\delta_0, K >0 $ such that
$$
a(t,x)\geq \delta_0, \quad \forall \, \omega, t,x.
$$
and
\begin{equation*}
 | a(t,\cdot)|_{C^2(I)}+| b(t,\cdot)|_{C^2(I)}+|c(t,\cdot)|_{C^2(I)} <K, \quad \forall \, \omega,t.
\end{equation*} 
(iii)  $\xi(t,x)$ is a  bounded $\mathcal{P}\times\mathcal{B}(I)$-measurable function, that is, 
\begin{equation*} 
\| \xi(t,\cdot) \|_\infty\leq K,  \quad  \quad \forall\,\, \omega,t.
\end{equation*}
\end{assumption}

\begin{thm} \label{theorem_general_lambda}
Let  $0\leq \lambda <\frac{1}{2}$, Assumption \ref{ass coeff} hold, $u_0\geq0$, and $u_0\in U_{p,\theta}^{1/2-\kappa}(I)$ for some $\kappa\in (\lambda, 1/2)$. Suppose that $p$ and $\theta$ satisfy

\begin{equation} \label{condition_of_constant_1}
 0< \theta  \leq 1+p(\frac{1}{2}+\kappa), \quad \quad p>\max\left(\frac{6}{1-2\kappa},\frac{2\lambda\theta}{1-2\lambda}\right).
\end{equation}
   Then,   the equation
 \begin{equation} \label{main_equation_1}
du = (au_{xx}+bu_x+cu)  \,dt+\sum_{k=1}^{\infty} \xi |u|^{1+\lambda} \eta_k dw^k_t, \quad  0<t <\infty; \,\,u(0,\cdot) = u_0
\end{equation} 
has a unique solution (in the sense of distributions) $u \in \mathfrak{H}_{p,\theta,loc}^{1/2-\kappa}(I,\infty)$.
Furthermore, for any $\alpha, \beta$ satisfying  \eqref{condition_for_alpha_beta_2} and  $\delta \in [0, \frac{1}{2}-\kappa-2\beta-\frac{1}{p})$,  we have  \begin{equation} \label{embedding_in_thm}
|\psi^{-\frac{1}{2}-\kappa-\frac{1}{p}+\frac{\theta}{p}-\delta}u|_{C^{\alpha-\frac{1}{p}}([0,T] ; C^{\frac{1}{2}-\kappa-2\beta-\frac{1}{p}-\delta} (I))} < \infty,
 \end{equation}
for all $T<\infty$ (a.s.).
\end{thm}

\begin{remark}
Note that $u_0$ is a $H^{1/2-\kappa-2/p}_{p,\theta+2-p}(I)$-valued random variable, and it is possible $1/2-\kappa-2/p< 0$. In this case, 
by $u_0\geq 0$ we mean $(u_0, \phi)\geq 0$ for any non-negative $\phi\in C^{\infty}_c(I)$.

\end{remark}  

\begin{remark}
 Due to the condition $p>6(1-2\kappa)^{-1}$ in \eqref{condition_of_constant_1}, one can always choose $\alpha, \beta$ such that  \eqref{condition_for_alpha_beta_2}  holds and the space H\"older exponent  $1/2-\kappa-2\beta-1/p$ in \eqref{embedding_in_thm} is positive.

According to the choice of $\alpha, \beta, \kappa$ one can obtain various  H\"older regularity. For instance, taking $\alpha, \beta$ sufficiently close to $1/p$ and $\delta=0$, from \eqref{embedding_in_thm} we get  the following space H\"older regularity: for any small $\varepsilon>0$,
\begin{equation}
      \label{11.26.1}
\sup_{t\leq T} |\psi^{-1/2-\kappa-1/p+\theta/p}u(t,\cdot)|_{C^{1/2-\kappa-3/p-\varepsilon}(I)}<\infty.
\end{equation}
Also choosing $\alpha, \beta$ sufficiently close to $1/4-\kappa/2-1/(2p)$ and $\delta=0$,  we get the following  time H\"older regularity: for any small $\varepsilon>0$,
\begin{equation}
     \label{11.26.2}
\sup_{x\in I} |\psi^{-1/2-\kappa-1/p+\theta/p}u(\cdot,x)|_{C^{1/4-\kappa/2-3/(2p)-\varepsilon}([0,T])}<\infty,
\end{equation}
for all $T<\infty$ (a.s.).

\end{remark}

  To take $p$ to $\infty$ in \eqref{embedding_in_thm} we need the following result.
  
\begin{thm}
    \label{thm unique}
   Let $u$ be taken from Theorem \ref{theorem_general_lambda}. Let $\theta', p'$ be constants such that \eqref{condition_of_constant_1} also holds
    with $(p',\theta')$,
    $(p',\theta)$, and  $(p,\theta')$. Assume $u_0\in U^{1/2-\kappa}_{p',\theta'}(I)$ so that 
   equation  \eqref{main_equation_1}   has the unique solution $\bar{u}\in \mathfrak{H}^{1/2-\kappa}_{p',\theta',loc}(I, \infty)$.   
    Then,  we have $u=\bar{u}$ in $\mathfrak{H}^{1/2-\kappa}_{p',\theta',loc}(I, \infty)$.    
    \end{thm}
    
    \begin{proof}
    By definition, there exists a sequence of bounded stopping times $\tau_n \uparrow \infty$ such that 
    $$
    u\in \mathfrak{H}^{1/2-\kappa}_{p,\theta}(I, \tau_n), \quad \bar{u}\in \mathfrak{H}^{1/2-\kappa}_{p',\theta'}(I, \tau_n).
    $$
    Without loss of generality assume $p' \geq p$. Put $\theta_0=\theta \vee \theta'$. Then by Remark \ref {remark 12.11.1},  
    $$
    u, \, \bar{u}\in     \mathfrak{H}^{1/2-\kappa}_{p, \theta_0}(I, \tau_n), \quad (n=1,2,\cdots),
    $$
    and therefore
    $$ u, \, \bar{u}\in     \mathfrak{H}^{1/2-\kappa}_{p, \theta_0,loc}(I, \infty).
    $$
    Note $U^{1/2-\kappa}_{p, \theta}(I) \subset U^{1/2-\kappa}_{p, \theta_0}(I)$. Thus, by Theorem \ref{theorem_general_lambda}, there exists a unique solution $\tilde{u} \in   \mathfrak{H}^{1/2-\kappa}_{p, \theta_0,loc}(I, \infty)$ to equation \eqref{main_equation_1}. Therefore, by the uniqueness, we conclude 
    $u=\bar{u}=\tilde{u}$.
    The theorem is proved.
        \end{proof}

      \begin{corollary}
    Suppose $u_0\in  U^{1/2-\lambda}_{p,\theta}(I)$ for any $p>2$ and $\theta>0$. Then for any  $T\in (0,\infty)$ and sufficiently small $\varepsilon>0$,     (a.s.)
    \begin{equation}
         \label{11.26.3}
    \psi^{-1/2-\kappa}u \in C^{\frac{1}{4}-\frac{\kappa}{2}-\varepsilon, \frac{1}{2}-\kappa-\varepsilon}_{t,x}([0,T]\times I), \quad \forall\, \kappa\in (\lambda, 1/2).
    \end{equation}
 In particular, for any $T>$ and $\varepsilon>0$, 
     $$
     \psi^{-1/2-\lambda}u \in C^{\frac{1}{4}-\frac{\lambda}{2}-\varepsilon, \frac{1}{2}-\lambda-\varepsilon}_{t,x}([0,T]\times I), \quad (a.s.),
     $$
     and 
     $$
     \sup_{t\leq T} |u(t,x)| \leq N(\omega,\varepsilon) \rho(x)^{1-\varepsilon}, \quad (a.s.).
     $$
    \begin{proof}
    Note $U^{1/2-\lambda}_{p,\theta}(I)\subset U^{1/2-\kappa}_{p,\theta}(I) $ for any $\kappa >\lambda$. Take $\theta=1$.  By Theorem \ref{theorem_general_lambda} for each sufficiently large $p$ there is a unique solution in   $\mathfrak{H}_{p,1,loc}^{1/2-\kappa}(I,\tau)$, and by Theorem \ref{thm unique}  all the solutions coincide for all large $p$. Thus, we get \eqref{11.26.3} by taking  $\delta=0$ and $p\to \infty$  in \eqref{11.26.1} and \eqref{11.26.2}. The last assertion is obtained by taking $\kappa$ sufficiently close to $1/2$ in \eqref{11.26.3}.  For the second assertion it is enough to take $\kappa$ sufficiently close to $\lambda$ and use the  fact that  for any $\nu>0$, $|\psi^{-1/2-\lambda}u|_{C^{\nu}}\leq N  |\psi^{-1/2-\kappa}u|_{C^{\nu}} $. The corollary is proved.  
     \end{proof}

\end{corollary}

\section{Auxiliary results}

In this section we introduce some auxiliary results on the  equation
\begin{equation} \label{equation_g}
du(t,x) = (au_{xx}+bu_x+cu )dt
+ \sum_{k=1}^{\infty} g^k(u)dw_t^k,\quad t>0 \quad ;\, u(0,\cdot) = u_0,
\end{equation}
where $g^k(u)=g^k(\omega, t,x,u)$.

\vspace{2mm}

Below, we recall Assumption \ref{ass coeff}(i)-(ii) for reader's convenience:
\begin{assumption} \label{ass coeff-3}
(i)  The coefficients $a(t,x)$, $b(t,x)$, $c(t,x)$ are $\mathcal{P}\times\mathcal{B}(I)$-measurable functions depending on $(\omega,t,x)$.

(ii) There exists    constants  $\delta_0, K >0 $ such that
$$
a(t,x)\geq \delta_0, \quad \forall \, \omega, t,x.
$$
and
\begin{equation*}
 | a(t,\cdot)|_{C^2(I)}+| b(t,\cdot)|_{C^2(I)}+|c(t,\cdot)|_{C^2(I)} <K, \quad \forall \, \omega,t.
\end{equation*}
\end{assumption}

The following is a special result of \cite[Theorem 2.14]{KK2004}.

\begin{thm}[linear case] \label{general_g_linear}
Let $\tau \leq T$, $|\gamma|\leq 2$, $p\geq2$, $\theta\in (0,p)$,    and $g(u)=g$ be independent of $u$.  Suppose Assumption \ref{ass coeff-3} holds,  $g\in \mathbb{H}_{p,\theta}^{\gamma+1}(I,\tau,l_2)$ and $u_0\in U_{p,\theta}^{\gamma+2}(I)$. Then  equation \eqref{equation_g} admits a unique solution $u$ in $\mathfrak{H}_{p,\theta}^{\gamma+2}(I,\tau)$, and for this solution 
\begin{equation*}
\| u \|_{\mathfrak{H}_{p,\theta}^{\gamma+2}(I,\tau)} \leq N( \| g \|_{\mathbb{H}_{p,\theta}^{\gamma+1}(I,\tau,l_2)} + \| u_0 \|_{U_{p,\theta}^{\gamma+2}(I)}),
\end{equation*}
where  $N=N(p,\theta,\gamma,K,\delta_0, T)$.
\end{thm}

We remark that the above result is true for other $\gamma$ if appropriate smoothness of the coefficients is assumed.

To generalize Theorem \ref{general_g_linear} to the nonlinear case we assume the following.

\begin{assumption}[$\tau$] \label{assumption_for_g}
 (i) For any  $u\in H_{p,\theta-p}^{\gamma+2}(I)$, the function $g(t,x,u)$  is an $H_{p,\theta}^{\gamma+1}(I,l_2)$-valued  predictable function, and $g(0)\in \bH^{\gamma+1}_{p,\theta}(\tau,I,l_2)$, where $\tau$ is a stopping time. 
 
(ii)  For any $\varepsilon>0$, there exists a constant $N_{\varepsilon}$ such that
\begin{equation} \label{eq_assumption_for_g}
\| g(t,u)-g(t,v) \|_{H^{\gamma+1}_{p,\theta}(I,l_2)}  \leq \varepsilon \| u-v \|_{H_{p,\theta-p}^{\gamma+2}(I)} + N_{\varepsilon}\|u-v\|_{H^{\gamma+1}_{p,\theta}(I)}
\end{equation}
 for any $u,v\in H_{p,\theta-p}^{\gamma+2}(I)$, $\omega\in \Omega$ and $t\in [0, \tau]$.

\end{assumption}

We emphasize that  we require \eqref{eq_assumption_for_g}  only for $t\leq \tau$, not for all $t$.

\begin{thm}[nonlinear case] \label{general_for_g}
 Let $|\gamma|\leq 2$, $p\geq2$, $\tau\leq T$, and  $\theta\in (0,p)$. Suppose Assumption \ref{ass coeff-3} and Assumption \ref{assumption_for_g}($\tau$) hold.  Then, for any $u_0\in U_{p,\theta}^{\gamma+2}(I)$,   equation \eqref{equation_g}  has a unique solution $u$ in $\mathfrak{H}_{p,\theta}^{\gamma+2}(I,\tau)$, and  for this solution 
\begin{equation} \label{estimate_g}
\| u \|_{\mathfrak{H}_{p,\theta}^{\gamma+2}(I,\tau)} \leq N( \| g(0) \|_{\mathbb{H}_{p,\theta}^{\gamma+1}(I,\tau,l_2)} + \| u_0 \|_{U_{p,\theta}^{\gamma+2}(I)}),
\end{equation} 
where  $N=N(p,\theta, \gamma,\delta_0, K,T)$.

\end{thm}
\begin{proof}

 We follow the proof of  \cite[Theorem 5.1]{kry1999analytic}.  The main difference is that in \cite{kry1999analytic} the result is proved on $\bR^d$ in classical Sobolve spaces without weights.  

{\bf{Step 1}}.  Assume that $\tau = T$.  Note that for any $v\in\mathfrak{H}_{p,\theta}^{\gamma+2}(I,T)$,  $g(v)\in \bH^{\gamma+1}_{p,\theta}(I,T,l_2)$. Indeed,
by \eqref{eq_assumption_for_g},
  \begin{eqnarray*}
 \|g(v)\|_{\bH^{\gamma+1}_{p,\theta}(I,T,l_2)}&\leq& \|g(v)-g(0)\|_{\bH^{\gamma+1}_{p,\theta}(I,T,l_2)}+\|g(0)\|_{\bH^{\gamma+1}_{p,\theta}(I,T,l_2)}\\
 &\leq&\|v\|_{\bH^{\gamma+2}_{p,\theta-p}(I,T)}+K\|v\|_{\bH^{\gamma+1}_{p,\theta}(I,T)} +\|g(0)\|_{\bH^{\gamma+1}_{p,\theta}(I,T,l_2)}\\
 &\leq& N \|v\|_{\mathfrak{H}_{p,\theta}^{\gamma+2}(I,T)}+\|g(0)\|_{\bH^{\gamma+1}_{p,\theta}(I,T,l_2)} <\infty.
 \end{eqnarray*}
 In the last inequality, we used   the relations
 $$
 \|v\|_{\bH^{\gamma+1}_{p,\theta}(I,T)}\leq N \|v\|_{\bH^{\gamma+2}_{p,\theta-p}(I,T)}\leq N\|v\|_{\mathfrak{H}_{p,\theta}^{\gamma+2}(I,T)}.
 $$ 
Thus, according to  Theorem \ref{general_g_linear}, for each $v\in\mathfrak{H}_{p,\theta}^{\gamma+2}(I,T)$, we can define $\mathcal{R}v$ as the unique solution to  the equation
\begin{equation*}
du(t,x) =(au_{xx}+bu_x+cu) dt+ \sum_{k=1}^{\infty} g^k(t,x,v) dw_t^k,\quad t>0 \quad ; u(0,\cdot)=u_0.
\end{equation*}
By   Theorem \ref{general_g_linear}, for  each $v,w\in \mathfrak{H}_{p,\theta}^{\gamma+2}(I,T)$ and $\varepsilon>0$,
\begin{equation} \label{contraction}
\begin{aligned}
\| \mathcal{R}v - \mathcal{R}w \|^p_{\mathfrak{H}_{p,\theta}^{\gamma+2}(I,T)} \leq N \| g(v) - g(w) \|^p_{\mathbb{H}_{p,\theta}^{\gamma+1}(I,T,l_2)} \\
\leq N_0\varepsilon^p\| v-w \|^p_{\mathfrak{H}_{p,\theta}^{\gamma+2}(I,T)} + N_1(\varepsilon) \int_0^T \| v-w \|^p_{\mathfrak{H}_{p,\theta}^{\gamma+2}(I,t)} dt.
\end{aligned}
\end{equation}
Observe that for the second inequality above we used  \eqref{eq_assumption_for_g} and \eqref{eqn 11.27.3}.
Here $N_0=N_0(p,\gamma,\theta,\delta_0, K,T)$  and $N_1=N_1(\varepsilon,p,\gamma,\theta,\delta_0, K,T)$.\\  
Now we fix $\varepsilon$ so that $\Theta:=N_0\varepsilon^p<1/4$. By repeating \eqref{contraction}, we have
\begin{equation*}
\begin{aligned}
&\| \mathcal{R}^mv - \mathcal{R}^mw  \|^p_{\mathfrak{H}_{p,\theta}^{\gamma+2}(I,T)} \\
&\quad\leq  \sum_{k=0}^m \begin{pmatrix} m\\k \end{pmatrix} \Theta^{m-k}\frac{1}{k!}(TN_1)^k\| v-w \|^p_{\mathfrak{H}_{p,\theta}^{\gamma+2}(I,T)} \\
&\quad\leq 2^m\Theta^m\max_k\left( \frac{TN_1}{\Theta^k k!} \right)\| v-w \|^p_{\mathfrak{H}_{p,\theta}^{\gamma+2}(I,T)}\\
&\quad\leq \frac{1}{2^m}N_2 \| v-w \|^p_{\mathfrak{H}_{p,\theta}^{\gamma+2}(I,T)}.
\end{aligned}
\end{equation*}
In the second inequality above, we  used $\sum_{k=0}^m \begin{pmatrix} m\\k \end{pmatrix} = 2^m$.  Thus, if  $m$ is large enough, then $\mathcal{R}^m$ is a contraction mapping in $\mathfrak{H}_{p,\theta}^{\gamma+2}(I,T)$. This certainly gives the unique solvability of equation \eqref{equation_g} in   $\mathfrak{H}_{p,\theta}^{\gamma+2}(I,T)$ together with the desired estimate.

{\bf{Step 2}}. General case. Define
$$
\bar{g}(t, u)=1_{t\leq \tau} g(t,u).
$$
Then, by the assumption on $g$,  $\bar{g}$ satisfies \eqref{eq_assumption_for_g}  for all $t\leq T$. Applying  the result of  Step 1 with $\bar{g}$ and $T$ in place of $g$ and $\tau$, we obtain the  solution   $u\in \mathfrak{H}_{p,\theta}^{\gamma+2}(I,T)$ together with  estimate \eqref{estimate_g}.  Since $\tau\leq T$,  $u$ certainly becomes a solution in $\mathfrak{H}_{p,\theta}^{\gamma+2}(I,\tau)$ with $g$ (in place of $\bar{g}$).

To prove the uniqueness, and let $w\in \mathfrak{H}_{p,\theta}^{\gamma+2}(I,\tau)$ be a solution to equation \eqref{equation_g} with initial data $u_0$. Using Theorem \ref{general_g_linear} (linear case), we define $\tilde{u}\in \mathfrak{H}_{p,\theta}^{\gamma+2}(I,T)$ as the unique solution to 
\begin{equation}
    \label{eqn 12.9.7}
d\tilde{u}=(a\tilde{u}_{xx}+b\tilde{u}_x+c\tilde{u})dt + \sum_{k=1}^{\infty} \bar{g}^k (w)  dw^k_t, \,t>0\, ; \quad \tilde{u}(0,\cdot)=u_0.
\end{equation}
We emphasize that in \eqref{eqn 12.9.7} we have $\bar{g}(w)$, not $\bar{g}(\tilde{u})$.  Then for each fixed $\omega$, $v=w-\tilde{u}$ satisfies the parabolic equation 
$$v_t=av_{xx}+bv_x+cv, \, \quad 0<t\leq \tau, \, \quad v(0,\cdot)=0.$$
Thus, due to the uniqueness of the deterministic equation (see  \cite[Theorem 2.10]{KK2}), we conclude $\tilde{u}=w$ for $t\leq \tau$.  Consequently one can replace $\bar{g}(w)$ in 
\eqref{eqn 12.9.7} by $\bar{g}(\tilde{u})$. It follows $\tilde{u}=u$ due to the uniqueness result in $\mathfrak{H}_{p,\theta}^{\gamma+2}(I,T)$  proved in Step 1, and the uniqueness also follows. The theorem is proved.
\end{proof}

\section{The case: \texorpdfstring{$\lambda=0$}{Lg}}

In this section we introduce an extension  of Theorem \ref{theorem_general_lambda} when $\lambda=0$. 

For this, we  consider  the following equation
\begin{equation} \label{equation_lambda_0}
\begin{aligned}
&du = (au_{xx}+bu_{x}+cu)dt 
+\sum_{k=1}^{\infty} \xi h\eta_k dw_t^k,\, 0<t\leq \tau;\,u(0,\cdot) = u_0,
\end{aligned}
\end{equation}
where $\xi=\xi(\omega,t,x)$, $h=h(u)=h(\omega,t,x,u)$ and $\eta_k = \eta_k(x)$.

\begin{assumption}
   \label{coeff again}
The coefficients $a(t,x)$, $b(t,x)$, $c(t,x)$ are $\mathcal{P}\times\mathcal{B}(I)$-measurable, and 
there exist    constants  $\delta_0, K >0 $ such that
$$
a(t,x)\geq \delta_0, \quad \forall \, \omega, t,x.
$$
and
$$
 | a(t,\cdot)|_{C^2(I)}+| b(t,\cdot)|_{C^2(I)}+|c(t,\cdot)|_{C^2(I)} <K, \quad \forall \, \omega,t.
$$
\end{assumption}
   
\begin{assumption} [$\tau,s$]\label{assumption_for_h}
(i) The function $h=h(\omega,t,x,u)$ is $\mathcal{P}\times\mathcal{B}(I)\times\mathcal{B}(\mathbb{R})$-measurable, and 
there exists a constant $K'$ such that for any $\omega,t,x,u,v$ 
\begin{equation} \label{eq_assumption_for_h}
\begin{aligned}
| h(t,x,u)-h(t,x,v) |  \leq K'| u-v |.
\end{aligned}
\end{equation}
(ii) The function $\xi=\xi(\omega, t,x)$ is $\mathcal{P}\times\mathcal{B}(I)$-measurable, and 
there exist $\theta'<s$ and $K''>0$ such that
\begin{equation} \label{2s_bound_of_xi}
\| \xi(t,\cdot) \|_{L_{2s,\theta'}(I)}\leq K'',\quad \forall\, \omega \in \Omega, \,\, t\in [0,\tau]\cap [0,\infty),
\end{equation}
where $\| \xi(t,\cdot) \|_{L_{\infty,\theta'}(I)} = \sup_{x\in I}|\xi(t,x)|$.
\end{assumption}

With help of Remark \ref{remark 12.11.1}, we may assume $\theta'>0$. Note that if $h(u)=|u|$ or $h(u)=u$ then \eqref{eq_assumption_for_h} is obviously satisfied.  Thus,  if $\lambda=0$, then  Theorem \ref{theorem_general_lambda}  is a particular case of the following result.


\begin{thm} \label{theorem_lambda_0}
Let  $\tau \leq T$, $p\geq2$, $\theta\in (0,p)$,  $\kappa\in (0,1/2)$ and $s>1$ such that
\begin{equation*}
 \frac{1}{2\kappa}<s\leq \infty,\quad p\geq\frac{2s}{s-1} \quad \quad (\frac{2 \infty}{\infty}:=2).
\end{equation*} 
Suppose $u_0\in U_{p,\theta}^{1/2-\kappa}(I)$ and Assumptions \ref{coeff again} and \ref{assumption_for_h}($\tau,s$) hold.
Then,  equation \eqref{equation_lambda_0}
has a unique solution $u$ in $\mathfrak{H}_{p,\theta}^{1/2-\kappa}(I,\tau)$, and for this   solution 
\begin{equation} \label{estimate_lambda_0}
\| u \|_{\mathfrak{H}_{p,\theta}^{1/2-\kappa}(I,\tau)} \leq N\| u_0 \|_{U_{p,\theta}^{1/2-\kappa}(I)},
\end{equation} 
where $N=N(p,\kappa, \theta,K,K',K'',\delta_0,T)$. Furthermore, if
\begin{gather*} 
0<\theta\leq\left(\kappa+\frac{1}{2}\right)p+1,\\
\frac{1}{p}<\alpha<\beta<\frac{1}{2}, \quad \frac{1}{2}-\kappa-2\beta-\frac{1}{p}>0,\\
0\leq\delta <\frac{1}{2}-\kappa-2\beta-\frac{1}{p},
\end{gather*}
then  (a.s.)
\begin{equation} \label{embedding_in_thm_2}
|\psi^{-\frac{1}{2}-\kappa-\frac{1}{p}+\frac{\theta}{p}-\delta}u|^p_{C^{\alpha-\frac{1}{p}}([0,\tau];C^{\frac{1}{2}-\kappa-2\beta-\frac{1}{p}-\delta} (I))} < \infty.
\end{equation}
\end{thm}
\begin{proof}
Put $\gamma= -\frac{3}{2}-\kappa$. Then, obviously $\gamma+2=1/2-\kappa$. Thus, to prove the theorem it is enough to check that the assumptions in  Theorem \ref{general_for_g} hold with $\gamma=-3/2-\kappa$.

Note that, due to Assumption  \ref{coeff again}, all the assumptions in Theorem \ref{general_for_g} holds  if 
$$g(u):= (\xi h(u)\eta_1, \xi h(u)\eta_2, \cdots)$$
satisfies \eqref{eq_assumption_for_g}:  for any given $\varepsilon>0$, there exists a constant $N$ such that 
\begin{equation} \label{claim}
\begin{aligned}
\|g(t,x,u) - g(t,x,v)&\|_{H^{-1/2-\kappa}_{p,\theta}(I,l_2)} \\
&\leq \varepsilon\| u-v \|_{H_{p,\theta-p}^{1/2-\kappa}(I)} + N\| u-v \|_{H_{p,\theta}^{-1/2-\kappa}(I)},
\end{aligned}
\end{equation}
for all $\omega \in \Omega,t\leq \tau$ and  $u,v\in H_{p,\theta-p}^{1/2-\kappa}(I)$. 

Below by $g(u)(e^nx)$ and $h(u)(x)$ we mean $g(t,e^nx, u(e^nx))$ and $h(t,x,u(x))$ respectively. 
By definitions of $L_p$-norms, 
\begin{align}
&\|g(u) - g(v)\|^p_{H_{p,\theta}^{-1/2-\kappa}(I,l_2)} \nonumber\\
&\quad= \sum_{n}e^{n\theta} \| \zeta_{-n}(e^n\cdot) [g(u)(e^n\cdot)-g(v)(e^n\cdot)] \|^p_{H_{p}^{-1/2-\kappa}(I,l_2)}\nonumber\\
&\quad= \sum_n e^{n\theta} \int_{\mathbb{R}} |(1-\Delta)^{(-1/2-\kappa)/2}\left([\zeta_{-n}(g(u)-g(v))(e^n\cdot)]\right)(x)|_{l_2}^p dx.   \label{eqn 12.5.1}
\end{align}
Put $w=u-v$. Then using \eqref{convolution}, changing of variables $e^nx \to x$, and Parseval's identity,   we have
\begin{equation}   \label{computation_about_l2}
\begin{aligned}
&|(1-\Delta)^{(-1/2-\kappa)/2}\left(\zeta_{-n}(e^n\cdot)\left[g(u)(e^n\cdot)-g(v)(e^n\cdot)\right]\right)(x)|_{l_2}^2\\
& = \sum_{k} \left( \int_{\mathbb{R}} R_{\kappa}(x-y)\zeta_{-n}(e^ny)\big[g^k(u)(e^ny)-g^k(v)(e^ny) \big]  dy \right)^2 \\
& = e^{-2n}\sum_{k} \left( \int_{\mathbb{R}} R_{\kappa}(x-e^{-n}y)(\zeta_{-n} \xi)(y)\left[(h(u)-h(v))(y) \right]\eta_k(y)dy \right)^2  \\
& = e^{-2n}\int_{\mathbb{R}} R^2_{\kappa}(x-e^{-n}y)\zeta^2_{-n}(y) \xi^2(y)\left[h(t,y,u)-h(t,y,v)\right]^2 dy \\
& \leq N e^{-2n}\int_{\mathbb{R}} R^2_{\kappa}(x-e^{-n}y)\xi^2(y)|w(y)|^2\zeta^2_{-n}(y) dy \\
& = Ne^{-n}\int_{\mathbb{R}} R^2_{\kappa}(y)\xi^2(e^n(x-y))|w(e^n(x-y))|^2 \zeta^2_{-n}(e^n(x-y)) dy.
\end{aligned}
\end{equation}
 Assume $s<\infty$, and put $r:=s/(s-1)$. Then, since we assume $2\kappa s>1$, we have $(1-2\kappa)r<1$, and therefore   
 $\|R_{\kappa}\|_{L_{2r}(\bR)}<\infty$ by \eqref{eqn 12.14.1}. 
  Coming back to \eqref{eqn 12.5.1}, and then using 
H\"older's inequality,
\begin{equation}  \label{main_computation_2}
\begin{aligned}
&\|g(u) - g(v)\|^p_{H_{p,\theta}^{-1/2-\kappa}(I,l_2)}  \\
&\quad\leq N\sum_{n}e^{n(\theta-\frac{p}{2})} \int_{\mathbb{R}} \Big(\int_{\mathbb{R}} R^2_{\kappa}(y)\xi^2(t,e^n(x-y)) \\
&\quad\quad\quad\quad\quad\quad\times |w(e^n (x-y))|^2\zeta_{-n}^2(e^n(x-y)) dy \Big)^\frac{p}{2} dx  \\
&\quad\leq N \sum_{n}e^{n(\theta-\frac{p}{2}(1+\frac{\theta'}{s}))}  \left(e^{n\theta'}\int_{\mathbb{R}} \xi^{2s}(t,e^ny)\zeta_{-n}^{s}(e^ny) dy \right)^{\frac{p}{2s}}\\
&\quad\quad\quad\quad\times\int_{\mathbb{R}}\left(\int_{\mathbb{R}} R^{2r}(y)|w(e^n (x-y))|^{2r}\zeta_{-n}^{r}(e^n (x-y)) dy \right)^\frac{p}{2r} dx.
\end{aligned}
\end{equation}
Observe that for each fixed $n\in\bZ$, we have
\begin{equation} \label{during_proof_2}
\begin{aligned}
e^{n\theta'}\int_{\bR}&|\xi(t,e^ny)|^{2s}\zeta_{-n}^s(e^ny)dy \leq \sum_{n} e^{n\theta'}\int_{\bR}|\xi(t,e^ny)|^{2s}\zeta_{-n}^s(e^ny)dy \\
&\leq N(\zeta)\sum_{n} e^{n\theta'}\int_{\bR}|\xi(t,e^ny)|^{2s}\zeta_{-n}^{2s}(e^ny)dy = \| \xi(t,\cdot) \|_{L_{2s,\theta'}(I)}^{2s}.
\end{aligned}
\end{equation}
By Remark \ref{indep_zeta}, \eqref{during_proof_2} and Mink\"owski's inequality(recall $p\geq2r$), we have
\begin{equation} \label{during_proof_1}
\begin{aligned}
&\sum_{n}e^{n(\theta-\frac{p}{2}(1+\frac{\theta'}{s}))}  \left(e^{n\theta'}\int_{\mathbb{R}} |\xi(t,e^ny)|^{2s}\zeta_{-n}^{s}(e^ny) dy \right)^{\frac{p}{2s}}\\
&\quad\quad\times\int_{\mathbb{R}}\left(\int_{\mathbb{R}} R^{2r}_{\kappa}(y)|w(t,e^n (x-y))|^{2r}\zeta_{-n}^{r}(e^n (x-y)) dy \right)^\frac{p}{2r} dx \\
& \leq N \| \xi(t,\cdot) \|_{L_{2s,\theta'}(I)}^p\sum_{n}e^{n(\theta-\frac{p}{2}(1+\frac{\theta'}{s}))} \\
&\quad\quad\times \int_{\mathbb{R}}\left(\int_{\mathbb{R}} R^{2r}_{\kappa}(y)|w(t,e^n (x-y))|^{2r}\zeta_{-n}^{r}(e^n(x-y)) dy \right)^\frac{p}{2r} dx  \\
&\leq  N \| \xi(t,\cdot) \|_{L_{2s,\theta'}(I)}^p \|R_{\kappa}\|^p_{L_{2r}(\bR)}\sum_{n}e^{n(\theta-\frac{p}{2}(1+\frac{\theta'}{s}))} \\
&\quad\quad\times  \int_{\mathbb{R}} |w(e^ny)|^{p}\zeta_{-n}^{p}(e^ny) dy\\
&= N  \| w\|_{L_{p,\theta-(1+\theta'/s)p/2}(I)}^p, \end{aligned}
\end{equation}
where $N =  \| \xi(t,\cdot) \|_{L_{2s,\theta'}(I)}^p \|R_{\kappa}\|^p_{L_{2r}} \leq (K'')^p\|R_{\kappa}\|^p_{L_{2r}}< \infty$. \\
Also, by Corollary \ref{interpol_cor},
\begin{equation*}
\begin{aligned}
\| w \|_{L_{p,\theta-p/2(1+\theta'/s)}(I)}^p &\leq \varepsilon\| w(t,\cdot)\|_{H^{1/2-\kappa}_{p,\theta-p}(I)}^p + N_{\varepsilon}\| w(t,\cdot) \|_{H_{p,\theta}^{-1/2-\kappa}(I)}^p,
\end{aligned}
\end{equation*} 
since $\kappa\in (0,1/2)$ and $\theta-p<\theta-\frac{p}{2}(1+\frac{\theta'}{s})<\theta$.
Therefore, we have \eqref{claim} when $s\in(1,\infty)$. 

If $s = \infty$,   we can take out $\xi$ out of integral as $\| \xi \|_\infty$ in \eqref{main_computation_2} so that we get
\begin{equation} \label{main_calculation_bdd}
\begin{aligned}
&\|g(u)-g(v) \|^p_{H_{p,\theta}^{-1/2-\kappa}(I,l_2)} \\
&\quad\leq N \sum_{n}e^{n(\theta-\frac{p}{2})}  \int_{\mathbb{R}} |w(e^ny)|^{p}\zeta_{-n}^{p}(e^ny) dy = N\|u-v\|^p_{L_{p,\theta-p/2}(I)}.
\end{aligned}
\end{equation}
 By Corollary \ref{interpol_cor} again,
\begin{equation} \label{main_calculation_bdd_2}
\begin{aligned} 
\| w \|^p_{L_{p,\theta-p/2}(I)} \leq \varepsilon \| w \|^p_{H^{1/2-\kappa}_{p,\theta-p}(I)}+N_{\varepsilon} \| w \|^p_{H^{-1/2-\kappa}_{p,\theta}(I)}.
\end{aligned}
\end{equation}
Therefore, we get \eqref{claim}, and the first claim of the theorem is proved.

Finally, \eqref{embedding_in_thm_2} is a consequence of  \eqref{embedding_remark}
 and \eqref{estimate_lambda_0}.
The theorem is proved.
\end{proof}



The following result, maximum principle, will be used in the next section.

\begin{thm}[Maximum principle]
\label{thm maximum}
Suppose  $p\geq 2$, $\theta\in (0,p)$,   $\kappa\in (0,1/2)$,  $u_0\in U_{p,\theta}^{1/2-\kappa}(I)$, $u_0\geq 0$ and
$$
\sup_{\omega} \sup_{t\leq \tau} \|\xi(t,\cdot)\|_{\infty} <\infty,
$$
where $\tau \leq T$ is a stopping time.  Let  $u \in \mathfrak{H}_{p,\theta}^{1/2-\kappa}(I,\tau)$ be the unique solution (taken from Theorem \ref{theorem_lambda_0}) to the equation
$$
du = (au_{xx}+bu_{x}+cu) \,dt + \sum_{k=1}^{\infty} \xi u\eta_{k} \,dw_t^k,\quad t\leq \tau;\quad u(0,\cdot) = u_0.
$$
Then $u(t,\cdot)\geq 0$ for all $t\leq \tau$ (a.s.).
\end{thm}

\begin{proof}
Considering $\bar{\xi}:=\xi 1_{t\leq \tau}$ in place of $\xi$, we may assume $\tau\equiv T$. By the  existence result (for $\bar{\xi}$) in $\mathfrak{H}_{p,\theta}^{1/2-\kappa}(I,T)$ and the uniqueness result in  
$\mathfrak{H}_{p,\theta}^{1/2-\kappa}(I,\tau)$, there is $\bar{u}\in \mathfrak{H}_{p,\theta}^{1/2-\kappa}(I,T)$ such that $u(t)=\bar{u}(t)$ for $t\leq \tau$. 
Take a sequence of   functions $u^n_0\geq 0$ such that $ u^n_0\in U^1_{p,\theta}(I)$, and $u^n_0 \to u_0$ in $U^{1/2-\kappa}_{p,\theta}(I)$ (see e.g. the proof of \cite[Theorem 1.19]{kry1999weighted} for detail). For $m=1,2,\cdots$, put
\begin{gather*}
\boldsymbol{\eta}=(\eta_1, \eta_2, \cdots), \quad \boldsymbol{\eta}_m=(\eta_1,\eta_2,\cdots,\eta_m,0,0,\cdots), \\
g(u)=\xi u\boldsymbol{\eta}=(\xi u \eta_1, \xi u\eta_2, \cdots),\\
g_m(u)=\xi u\boldsymbol{\eta}_m=(\xi u \eta_1, \xi u\eta_2, \cdots,\xi  u \eta_m,0,0,\cdots).
\end{gather*}
Then the calculations in \eqref{computation_about_l2} show that for any $u,v\in H^{1/2-\kappa}_{p,\theta-p}(I)$,
\begin{equation}
   \label{eqn 12.6.5}
\|g_m(u)-g_m(v)\|_{H^{-1/2-\kappa}_{p,\theta}(I,l_2)} \leq \|g(u)-g(v)\|_{H^{-1/2-\kappa}_{p,\theta}(I,l_2)}.
\end{equation}
Thus, \eqref{claim}  and  all the assumptions in  Theorem \ref{general_for_g}  hold  with $g_m(u)$ in place of $g(u)$. Hence we can define $u_m \in \mathfrak{H}_{p,\theta}^{1/2-\kappa}(I,T)$ as the unique solution to
$$
du=(au_{xx}+bu_x+cu)dt+\sum_{k=1}^{\infty} g^{k}_m (u)dw^k_t, \,\, 0<t\leq T\, ; \quad u(0,\cdot)=u^m_0.
$$
To finish the proof,  we show that  $u_m \to u$ in $\mathfrak{H}_{p,\theta}^{1/2-\kappa}(I,T)$ and $u_m \geq 0$.
Note  that $v_m:=u-u_m$ satisfies
\begin{align*}
&dv=(av_{xx}+bv_x+cv)dt+\sum_{k=1}^{\infty} (g^k(u)-g^{k}_m (u_m))dw^k_t,\quad0<t\leq T\\
&v(0,\cdot)=u_0(\cdot)-u^m_0(\cdot),
\end{align*}
and
$$
g(u)-g_{m}(u_m)=g_m(u)-g_m(u_m)+ \xi u ( \boldsymbol{\eta}-\boldsymbol{\eta}_m) .
$$ 
Thus, by Theorem \ref{general_g_linear} with $\tau=t$, \eqref{eqn 12.6.5} and \eqref{claim},  we get  for any  $t \leq T$,
\begin{eqnarray*}
&&\| u-u_m \|^p_{\mathfrak{H}_{p,\theta}^{1/2-\kappa}(I,t)}\\
&&  \quad   \leq N\|u_0-u^m_0\|^p_{U^{1/2-\kappa}_{p,\theta}(I)}+N\|g_m(u)-g_m(u_m))\|^p_{\mathbb{H}_{p,\theta}^{-1/2-\kappa}(I,t,l_2)} \\
&&\quad\quad+ \| \xi u ( \boldsymbol{\eta}-\boldsymbol{\eta}_m)\|^p_{\mathbb{H}_{p,\theta}^{-1/2-\kappa}(I,t,l_2)}\\
&&\quad\leq N\|u_0-u^m_0\|^p_{U^{1/2-\kappa}_{p,\theta}(I)}+ \frac{1}{2}\| u-u_m \|^p_{\mathbb{H}^{1/2-\kappa}_{p,\theta-p}(I,t)} \\
&&\quad\quad+N\| u-u_m \|^p_{\mathbb{H}^{-1/2-\kappa}_{p,\theta}(I,t)}+N \| \xi u ( \boldsymbol{\eta}-\boldsymbol{\eta}_m)\|^p_{\mathbb{H}_{p,\theta}^{-1/2-\kappa}(I,t,l_2)}.
\end{eqnarray*}
 Therefore, for any $t\leq T$,
\begin{eqnarray*}
&&\| u-u_m \|^p_{\mathfrak{H}_{p,\theta}^{1/2-\kappa}(I,t)} \leq  N\|u_0-u^m_0\|^p_{U^{1/2-\kappa}_{p,\theta}(I)}+ N\| u-u_m \|^p_{\mathbb{H}^{-1/2-\kappa}_{p,\theta}(I,t)} \\
&&\quad \quad \quad \quad  \quad +N \| \xi u ( \boldsymbol{\eta}-\boldsymbol{\eta}_m)\|^p_{\mathbb{H}_{p,\theta}^{-1/2-\kappa}(I,T,l_2)}.
\end{eqnarray*}
This, \eqref{eqn 11.27.3},  and Gronwall's inequality yield
\begin{equation} \label{during_proof_12}
\begin{aligned}
\| u-u_m\|^p_{\mathfrak{H}_{p,\theta}^{1/2-\kappa}(I,T)}\leq N&\|u_0-u^m_0\|^p_{U^{1/2-\kappa}_{p,\theta}(I)}\\
&+ N  \| \xi u ( \boldsymbol{\eta}-\boldsymbol{\eta}_m)\|^p_{\mathbb{H}_{p,\theta}^{-1/2-\kappa}(I,T,l_2)},
\end{aligned}
\end{equation}
where $N$ is independent of $m$ and $u$. \\
The calculations in  \eqref{eqn 12.5.1} and \eqref{computation_about_l2} certainly show (recall $\xi$ is bounded) that 
$\| \xi(t) u(t) ( \boldsymbol{\eta}-\boldsymbol{\eta}_m)\|^p_{H_{p,\theta}^{-1/2-\kappa}(I,l_2)}$ less than or equal to
$$
N \sum_{n=-\infty}^{\infty} e^{n(\theta-p/2)} \int_{\bR} \left( \sum_{k>m} \left[ \int_{\bR}R_{\kappa}(x-e^{-n}y)\zeta_{-n}(y) u(t,y) \eta_k(y) dy   \right]^2 \right)^{p/2} dx,
$$
 and by calculations in \eqref{during_proof_1}, the term above is bounded by $N\|u(t)\|^p_{H^{1/2-\kappa}_{p,\theta-p}(I)}$, uniformly in $m$. By the dominated convergence theorem,
 $$\| \xi(t) u(t) ( \boldsymbol{\eta}-\boldsymbol{\eta}_m)\|^p_{H_{p,\theta}^{-1/2-\kappa}(I,l_2)} \to 0$$
as $m\to \infty$ for almost all $(\omega,t)$, and  therefore using the dominated convergence theorem again we get  
 $$\| \xi u ( \boldsymbol{\eta}-\boldsymbol{\eta}_m)\|^p_{\mathbb{H}_{p,\theta}^{-1/2-\kappa}(I,T,l_2)} \to 0 \quad \text{as}\quad m\to \infty.
 $$ 
 
 Next, we show $u_m \geq 0$. Recall $u^m_0\in U^1_{p,\theta}(I)$. Since $u_m \in \bH^{1/2-\kappa}_{p,\theta-p}(I,T)\subset \bL_{p,\theta}(I,T)$, $\xi$ and $\eta_k$ are bounded and $g^k_m(u_m):=\xi \eta_k u_m 1_{k\leq m}$ is zero for $k>m$, we easily conclude 
 $$
 g_m(u_m)\in \bL_{p,\theta}(I,T,l_2),
 $$
  and therefore by   Theorem \ref{general_g_linear}  there is a unique solution $\bar{u}_m \in \mathfrak{H}_{p,\theta}^{1}(I,T)$  to the equation
  $$
  dv=(av_{xx}+bv_x+cv)dt+ \sum_{k=1}^{\infty} g^k_m(u_m)dw^k_t, \,  0<t\leq T\, ; \quad v(0,\cdot)=u^m_0.
  $$
  Consequently,  both $u_m$ and $\bar{u}_m$ are solutions to the above equations. Since $\mathfrak{H}_{p,\theta}^{1}(I,T) \subset \mathfrak{H}_{p,\theta}^{1/2-\kappa}(I,T)$,  the uniqueness result in $ \mathfrak{H}_{p,\theta}^{1/2-\kappa}(I,T)$   yields $u_m=\bar{u}_m$.
  Finally, to conclude $u_m \geq 0$, note that $u_m \in \mathfrak{H}_{p,\theta}^{1}(I,T)$ and it satisfies
  $$
  dv=(av_{xx}+bv_x+cv)dt+\sum_{k\leq m} \nu^k v \, dw^k_t, \, t>0 \, ; \quad v(0,\cdot)=u^m_0 \geq 0,
  $$
  where $\nu^k:=\xi \eta_k$ are bounded. Hence, by a classical maximum principle (see e.g.\cite[Theorem 1.1]{krylov2007maximum})  we conclude $u_m \geq 0$.
  
    We  remark that, since  $u_m \in \mathfrak{H}_{p,\theta}^{1}(I,T)$,   both $\phi u_m$ and $\phi D_xu_m$ are square integrable for any 
    $\phi \in C^{\infty}_c(I)$, and this is needed to apply \cite[Theorem 1.1]{krylov2007maximum}.
  The theorem is proved.
\end{proof}

\section{The case \texorpdfstring{$\lambda\in (0,1/2)$}{Lg}}

In this section we prove Theorem \ref{theorem_general_lambda} when $\lambda\in (0,1/2)$.
For reader's convenience we recall Assumption \ref{ass coeff} below.

\begin{assumption}[$\tau$]\label{ass coeff-1}
(i) The coefficients $a(t,x)$, $b(t,x)$, $c(t,x)$ are $\mathcal{P}\times\mathcal{B}(I)$-measurable functions depending on $(\omega,t,x)$.\\
(ii) There exists    constants  $\delta_0, K >0 $ such that
$$
a(t,x)\geq \delta_0, \quad \forall \, \omega,t,x.
$$
and
\begin{equation*}
\label{eqn 12.7.3}
 | a(t,\cdot)|_{C^2(I)}+| b(t,\cdot)|_{C^2(I)}+|c(t,\cdot)|_{C^2(I)} <K, \quad \forall \, \omega,t.
\end{equation*}
(iii)  $\xi(t,x)$ is  $\mathcal{P}\times\mathcal{B}(I)$-measurable and there exist a constant $K$ such that 
\begin{equation} 
      \label{eqn 12.10.8}
\| \xi(t,\cdot) \|_\infty\leq K,  \quad  \quad \forall\,\, \omega\in \Omega, \, t\leq \tau.
\end{equation}

\end{assumption}

\begin{lemma} \label{cut_off_lemma}
Choose $\lambda,\kappa,p,$ and $\theta$ so that
\begin{equation*}
\lambda>0,\quad\kappa\in(0,1/2),\quad p\geq2,\quad \mbox{and} \quad0<\theta<p.
\end{equation*}
Let  $\tau\leq T$ and Assumption \ref{ass coeff-1}($\tau$) hold. Assume $u_0\in U_{p,\theta}^{1/2-\kappa}(I)$ and $u_0\geq0$.
Then, for any positive integer $m$,  the equation
\begin{equation*} \label{eqn 12.7.1}
\begin{aligned}
du=  (au_{xx}+bu_{x}&+cu)dt \\
&+\sum_{k=1}^{\infty} \xi |(-m) \vee u\wedge m|^{1+\lambda}\eta_k dw_t^k,\,\, 0<t\leq \tau;\quad u(0,\cdot)=u_0
\end{aligned}
\end{equation*}
has a unique solution $u_m \in \mathfrak{H}_{p,\theta}^{1/2-\kappa}(I,\tau)$ and $u_m\geq 0$. Moreover, if \eqref{eqn 12.10.8} holds for all $t\in [0,\infty)$, then  $u_m\in\mathfrak{H}_{p,\theta}^{1/2-\kappa}(I,T')$ for any $T'<\infty$.

\end{lemma}

\begin{proof}

By mean-value theorem,
\begin{equation} \label{nonlinear}
|(-m)\vee u\wedge m|^{1+ \lambda} - |(-m)\vee v\wedge m|^{1+ \lambda} \leq (1+\lambda)(2m)^\lambda |u-v|.
\end{equation}

Thus, Assumption \ref{assumption_for_h}($\tau,s$) is satisfied with $s=\infty$. Consequently, by Theorem \ref{theorem_lambda_0}, the equation in the lemma has a unique solution $u_m\in\mathfrak{H}_{p,\theta}^{1/2-\kappa}(I,\tau)$.

 To show $u_m \geq 0$, we define
\begin{equation}
   \label{eqn 12.9.7_1}
\xi_m(t,x) = \frac{\xi(t,x) |(-m)\vee u_m(t,x)\wedge m|^{1+\lambda}}{u_m(t,x)}\quad \left(\frac{0}{0}:=0\right).
\end{equation}
Note that $\xi_m$ is a bounded, that is $\sup_{\omega}\sup_{t\leq \tau} \|\xi_m\|_{\infty}<\infty$. Also note that   $u_m$  satisfies
\begin{equation*}
\begin{aligned}
dv  = (av_{xx}+bv_{x}+cv)dt +\sum_{k=1}^{\infty} \xi_mv\eta_k dw_t^k,\quad t\leq \tau\, ;\quad v(0,\cdot)=u_0\geq 0.
\end{aligned}
\end{equation*}
Hence, by Theorem \ref{thm maximum} (maximum principle) we conclude  $u_m \geq 0$. 

Finally, assume that \eqref{eqn 12.10.8} holds for all $t$ and let $T'>T$.  Then by the previous result applied with $\tau=T'$, there exists a unique solution for $t\leq T'$. This obviously coincide with $u$ for $t\leq \tau$ due to the uniqueness result in $\mathfrak{H}_{p,\theta}^{1/2-\kappa}(I,\tau)$. The lemma is proved.

\end{proof}

\begin{remark}
    \label{remark 12.9.1}
Let \eqref{eqn 12.10.8} hold for all $t$, and $u_m$ be taken from the above Lemma.  Then, for any $\phi \in C^{\infty}_c(I)$
$$
M_t:=\sum_{k=1}^{\infty} \int^t_0 (\xi |u_m \wedge m|^{1+\lambda} \eta_k, \phi )dw^k_t, \quad t<\infty
$$
is a square integrable martingale, because
$$
\bE \sum_{k=1}^{\infty} \int^t_0 \left(\xi |u_m \wedge m|^{1+\lambda} \eta_k, \phi  \right)^2 ds \leq N \bE \sum_{k=1}^{\infty} \int^t_0 ( \eta_k, \phi )^2ds\leq N  \|\phi\|^2_{L_2(I)}<\infty.
 $$
\end{remark}

\begin{lemma} \label{weight_lemma}
Suppose Assumption \ref{ass coeff-1}($\tau$)(ii) holds, that is, 
$$
a\geq \delta_0, \quad  | a(t,\cdot)|_{C^2(I)}+| b(t,\cdot)|_{C^2(I)}+|c(t,\cdot)|_{C^2(I)} <K.
 $$
 Let $K_1 =3 K/(2\delta_0)$ and 
\begin{equation} \label{weight}
\psi(x) = -\cosh(K_1(2x-1))+ \cosh(K_1).
\end{equation}
Then $\psi(x)$ is comparable to $\rho(x)$, infinitely differentiable in $I$, concave in $I$, and 
\begin{equation} \label{condi_for_gen}
a(t,x)\psi''(x)+(2a_x(t,x)-b(t,x))\psi'(x)\leq 0.
\end{equation}
\end{lemma}
\begin{proof}
It is easy to check that 
$\psi(0)=\psi(1)=0,  \psi'(0)>0,  \psi'(1)<0$, $\psi(x)>0$ in $I$, and $\psi''(x)<0$ in $I$.  Therefore, it is sufficient to show \eqref{condi_for_gen}. Observe that
\begin{equation*}
\begin{aligned}
&a(t,x)\psi''(x)+(2a_x(t,x)-b(t,x))\psi'(x) \\
&\quad \leq \delta_0\psi''(x)+|(2a_x(t,x)-b(t,x))\psi'(x)| \\
&\quad \leq -4K_1^2\delta_0\cosh(K_1(2x-1)) + 6KK_1|\sinh(K_1(2x-1))| \\
&\quad = -6K_1K \Big(\cosh(K_1(2x-1))-|\sinh(K_1(2x-1))| \Big) \leq 0.
\end{aligned}
\end{equation*}
The lemma is proved.
\end{proof}

\begin{lemma} \label{key_lemma}
Let  \eqref{eqn 12.10.8} hold for all $t$, and $u_m$ be the solution  taken from Lemma \ref{cut_off_lemma}.  Also assume \eqref{condition_of_constant_1} holds. Then, for any $T<\infty$,
\begin{equation} \label{lemma}
\lim_{R\to\infty}\sup_{m}P(\{ \omega:\sup_{t\leq T,x\in I}| \psi^{-1/2-\kappa-\frac{1}{p}+\frac{\theta}{p}}(x)u_m(t,x) | \geq R \}) = 0.
\end{equation}
\end{lemma}

\begin{proof}
Without loss of generality, we use $\psi$  introduced in Lemma \ref{weight_lemma}. Fix $\alpha\in(0,1)$ so that
\begin{equation}
    \label{eqn 12.9.6}
\frac{\theta}{p}<\alpha<-1+\frac{1}{2\lambda}.
\end{equation}
Take $\phi\in C^\infty(\mathbb{R})$ such that $\phi = 0$ on $x<1/2$, $\phi = 1$ on $x>1$ and $\phi'(x)\geq0$. Define $\phi_k(x) = \phi(k\psi(x))$. Note that  $\phi_k$ are infinitely differentiable in $I$, $\phi_k(x) \to 1$ in $I$ as $k\to\infty$.  
Since $\alpha-1<0$ and $\psi''(x)\leq0$, we have
\begin{equation} \label{diff_twice}
\begin{aligned}
(\psi^\alpha &\phi(k\psi))'' = \alpha(\alpha-1)\psi^{\alpha-2}|\psi'|^2\phi_k + \alpha\psi^{\alpha-1}\psi''\phi_k\\
&+ 2\alpha\psi^{\alpha-1}\psi'\phi'(k\psi)k\psi'+\psi^\alpha\phi''(k\psi)(k\psi')^2
+\psi^\alpha\phi'(k\psi)k\psi'' \\
&\leq \alpha \psi^{\alpha-1}\psi''\phi_k +2\alpha k\psi^{\alpha-1}|\psi'|^2\phi'(k\psi)
  +k^2 \psi^\alpha |\psi'|^2 \phi''(k\psi).
\end{aligned}
\end{equation}
Recall $\sup_x(|a_{xx}|+|b_x|+|c|) \leq 3K$.   By It\^o's formula,  $v:=u^m e^{-3Kt}$ satisfies
\begin{equation*}
\begin{aligned}
dv=(av_{xx}+bv_x&+cv-3Kv)dt \\
&+\sum_{k=1}^{\infty} \xi e^{-3Kt} |u_m \wedge m|^{1+\lambda} \eta_k dw^k_t, \, t>0\, ; \quad v(0,\cdot)=u_0.
\end{aligned}
\end{equation*}
Multiplying by  $\psi^{\alpha} \phi_k$ to the above  equation (actually using $\psi^{\alpha}\phi_k$ as a test function), taking expectation (also see Remark \ref{remark 12.9.1}),  integrating by parts,  and then using \eqref{diff_twice} we get  for any  stopping time $\tau\leq T$, 

\begin{eqnarray*}
&&\mathbb{E}(e^{-3T K}u_m(\tau),\psi^{\alpha}\phi_k)\leq\,\mathbb{E}(e^{-3\tau K}u_m(\tau),\psi^{\alpha}\phi_k) \\
&=& \mathbb{E}(u_0, \psi^{\alpha}\phi_k) + \mathbb{E}\int_0^\tau \Big(u_m, ( a \psi^{\alpha} \phi_k)''-(b \psi^\alpha \phi_k)'+cu -3Ku \Big)e^{-3Kt} dt \\
&=& \mathbb{E}(u_0, \psi^{\alpha}\phi_k) + \mathbb{E}\int_0^\tau \Big( u_m, \, a(\psi^\alpha\phi_k)''+(2a_x-b)(\psi^\alpha\phi_k)'\\
&&\quad\quad\quad\quad\quad\quad\quad \quad \quad \quad \quad +(a_{xx}-b_x+c-3K)(\psi^\alpha\phi_k) \Big)e^{-3Ks}ds \\
&\leq& \bE(u_0, \psi^{\alpha}) + \mathbb{E}\int_0^\tau \int_{I} \bigg(\left[a\psi''+(2a_x-b)\psi' \right]\alpha \psi^{\alpha-1}\phi_k 
   \\
&&\quad\quad\quad\quad\quad\quad\quad\quad\quad\quad\quad\quad + 2a\alpha k \psi^{\alpha-1}|\psi'|^2\phi_k' + ak^2\psi^\alpha |\psi'|^2\phi_k'' \\
&&\quad\quad\quad\quad\quad\quad\quad\quad\quad\quad\quad\quad\quad + (2a_x-b)k\psi^\alpha\psi'\phi_k'\bigg)e^{-3Ks}u_m\,dxds.
\end{eqnarray*}
Now we use $a\psi''+(2a_x-b)\psi' \leq 0$, and
$$
|2a\alpha \psi^{\alpha-1}|\psi'|^2 \phi_k'|+ |(2a_x-b)\psi^{\alpha} \psi' \phi_k'|\leq N \psi^{\alpha-1}|\psi'| |\phi_k'|,
$$
and conclude for any $\tau\leq T$,
\begin{align*}
e^{-3KT} \bE (u_m(\tau,\cdot), \psi^{\alpha}\phi_k)&\leq N_0 +N k \bE\int^{\tau}_0 \int_I \psi^{\alpha-1}|\psi' | |\phi'(k\psi)|u_m dx dt \\
&\quad+Nk^2  \bE\int^{\tau}_0 \int_I \psi^{\alpha}|\psi'|^2 |\phi''(k\psi)|u_m dx dt.
\end{align*}
%
%
 Choose $q,\theta'\in\mathbb{R}$ so that
\begin{equation*}
\frac{1}{p}+\frac{1}{q} = 1,\quad\frac{\theta}{p}+\frac{\theta'}{q}=1.
\end{equation*}
By H\"older's inequality and the fact $|\psi'(x)|^q\leq N|\psi'(x)|$ for all $x\in I$, we get
\begin{equation*}
\begin{aligned}
&k\mathbb{E}\int_0^\tau \int_{I}  u_m\  \psi^{\alpha-1}|\psi'| |\phi'(k\psi)| dx ds  \\
&\quad=k\mathbb{E}\int_0^\tau \int_{I} ( \psi^{-1} u_m)    (\psi^{\alpha}|\psi'| \phi'(k\psi))   \psi^{\frac{\theta-1}{p}+\frac{\theta'-1}{q}} dx ds  \\
&\quad\leq Nk\| \psi^{-1}u_m \|_{\mathbb{L}_{p,\theta}(I,T)}\left( \int_I  |   \psi^{\alpha}\phi'(k\psi) |^q|\psi'|^{q} \psi^{\theta'-1} dx\right)^{1/q} \\
&\quad\leq N k^{1-\alpha-\frac{\theta'}{q}}M_m\left(\int_{I}  |  (k\psi)^\alpha \phi'(k\psi) |^q  (k\psi)^{\theta'-1}|k\psi'|dx\right)^{1/q}
 \\
&\quad\leq NM_m k^{\frac{\theta}{p}-\alpha}\left( \int_{\mathbb{R}_+}  
| \phi'(y)|^q y^{\alpha q+\theta'-1} dy \right)^{1/q} \\
&\quad \leq N M_m k^{\frac{\theta}{p}-\alpha},
\end{aligned}
\end{equation*}
where $M_m:= \| \psi^{-1}u_m \|_{\mathbb{L}_{p,\theta}(I,T)}$. Similarly,
\begin{equation*} 
\begin{aligned}
&k^2\mathbb{E}\int_0^\tau\int_{I}  |u_m  \psi^\alpha \phi''(k\psi)||\psi'|^2  dxds \\
&\quad=k^2\mathbb{E}\int_0^\tau\int_{I}  |\psi^{-1}u_m||\psi^{\alpha+1} \phi''(k\psi)||\psi'|^2  \psi^{\frac{\theta-1}{p}+\frac{\theta'-1}{q}}dxds \\
&\quad\leq Nk^2 \| \psi^{-1}u_m \|_{\mathbb{L}_{p,\theta}(I,T)} \left( \int_I (| \psi^{\alpha+1} \phi''(k\psi)||\psi'|^2 )^q \psi^{\theta'-1} dx \right)^{1/q} \\
&\quad\leq N M_m k^{1-\alpha-\frac{\theta'}{q}} \left(\int_{I}  | (k\psi)^{\alpha+1} \phi''(k\psi) |^q (k\psi)^{\theta'-1}|k\psi'| dx \right)^{1/q} \\
&\quad\leq NM_m k^{\frac{\theta}{p}-\alpha} \left( \int_{I}  |\phi''(y)|^q  y^{(\alpha+1)q+\theta'-1}dy \right)^{1/q}\\
&\quad \leq N M_m k^{\frac{\theta}{p}-\alpha}.
\end{aligned}
\end{equation*}
Note that the constant $N$ is independent of $m, k$. 
Therefore, for any stopping time $\tau\leq T$,
\begin{equation*}
\mathbb{E}(u_m(\tau),\psi^{\alpha}\phi_k) \leq N_0e^{3KT}+ N M_me^{3KT} k^{\frac{\theta}{p}-\alpha}.
\end{equation*}
By e.g. \cite[Theorem III.6.8]{diffusion},  for any $\gamma\in(0,1)$,
\begin{equation*}
\mathbb{E}\sup_{t\leq T}(u_m(t),\psi^{\alpha}\phi_k))^\gamma \leq e^{3KT\gamma}(N_0 + NM_m k^{\frac{\theta}{p}-\alpha})^{\gamma},
\end{equation*}
where $N_0, N$ are independent of $m,k$. 
Since $\alpha>\theta/p$, by monotone convergence theorem,
$$ \mathbb{E}\sup_{t\leq T}\| \psi^\alpha u_m(t) \|_{L_1(I)}^{1/2} \leq e^{2KT} N_0^{1/2}=:N^{1/2}_1.
$$
Thus,
\begin{equation} \label{stopping_time}
P(\sup_{t\leq T}\| \psi^\alpha(\cdot) u_m(t,\cdot) \|_{L_1(I)}\geq S) \leq N_1^{1/2}/\sqrt{S}.
\end{equation}
Recall that the $N_0$ is independent of $m$ and $S>0$.\\
Fix $m,S>0$. Define
$$ \tau_m(S) :=\inf\{ t\geq 0:\| \psi^\alpha(\cdot)u_m(t,\cdot) \|_{L_1(I)}\geq S \}.
$$
$\tau_m(S)$ becomes a stopping time due to Corollary \ref{embedding_corollary}. Put $s=1/(2\lambda)$ and $\theta'=\alpha+1$. Take $\xi_m$ from \eqref{eqn 12.9.7_1}, that is, since $u_m \geq 0$,
$$
\xi_m(t,x) = \frac{\xi(t,x) | u_m(t,x)\wedge m|^{1+\lambda}}{u_m(t,x)}\quad \left(\frac{0}{0}:=0\right).
$$
 Note that by \eqref{eqn 12.9.6},  $\theta'<s$, and for $t\leq \tau_m(S)$,
$$
\|\xi_m(t,\cdot)\|^{2s}_{2s,\theta'}\leq \int_{I} |\xi|^{2s} (|u_m(t)\wedge m|^{\lambda})^{2s} \psi^{\theta'-1}dx\leq \|\xi\|_{\infty} \|\psi^{\alpha} u_m(t)\|_{L_1(I)} \leq KS.
$$
Since, for any $T<\infty$,  $u_m$ satisfies
$$
du=(au_{xx}+bu_x+cu)dt+\sum_{k=1}^{\infty} \xi_m u\eta_k dw^k_t, \quad 0<t\leq \tau_m(S) \wedge T;\quad u(0,\cdot) = u_0.
$$
By Theorem \ref{theorem_lambda_0}, we get
\begin{equation*}
\| u_m \|^p_{\mathfrak{H}_{p,\theta}^{1/2-\kappa}(I, T\wedge \tau_m(S))} \leq N\| u_0 \|^p_{U_{p,\theta}^{1/2-\kappa}(I)}, 
\end{equation*}
where $N$ is independent of $m$ (of course, it depends on $S$ and $T$). Now, by \eqref{embedding_remark}, $\mathbb{E}\sup_{t\leq T\wedge \tau_m(S),x}|\psi^{-1/2-\kappa-\frac{1}{p}+\frac{\theta}{p}}(x)u_m(t,x)|^p$ is bounded by a constant which is independent of $m$.  Thus, by Chebyshev's inequality,
\begin{equation*}
\begin{aligned}
P(&\sup_{t\leq T,x\in I}|\psi^{-1/2-\kappa-\frac{1}{p}+\frac{\theta}{p}}u_m|\geq R)\\
&\leq P(\sup_{t\leq T \wedge \tau_m(S),x\in I}|\psi^{-1/2-\kappa-\frac{1}{p}+\frac{\theta}{p}}u_m | \geq R) + P(\tau_m(S)<T)  \\
&\leq \frac{N_S}{R} + P(\sup_{t\leq T}\| \psi^\alpha(\cdot) u_m(t,\cdot) \|_{L_1(I)}\geq S) \leq \frac{N_S}{R} +\frac{N}{\sqrt{S}},
\end{aligned}
\end{equation*}
where  $N_S$ is independent of $m,R$, and  $N$ is independent of $m,R,S$. By taking the supremum with respect to $m$,  letting $R\to\infty$ and $S\to\infty$ in order, we get \eqref{lemma}. The lemma is proved.

\end{proof}

\textbf{Proof of Theorem \ref{theorem_general_lambda}}. 

{\bf{Step 1}}. Uniqueness. Let $u,v\in\mathfrak{H}_{p,\theta,loc}^{1/2-\kappa}(I,\infty)$ be  two solutions to equation \eqref{main_equation_1}.  Then there exists a sequence of bounded stopping times $\tau_m \uparrow \infty$ such that 
\begin{equation}
    \label{eqn 12.9.11}
u,v \in \mathfrak{H}_{p,\theta}^{1/2-\kappa}(I,\tau_m).
\end{equation}  Fix $m$ and put $\tau'=\tau_m$. Since \eqref{condition_of_constant_1}, by Corollary \ref{embedding_corollary}, there exist $\varepsilon_1, \varepsilon_2>0$ such that 
\begin{equation}
        \label{eqn 12.9.10}
\mathbb{E}|u|^p_{C^{\varepsilon_1}([0,\tau']; C^{\varepsilon_2}(I))}+ \mathbb{E}|v|^p_{C^{\varepsilon_1}([0,\tau']; C^{\varepsilon_2}(I))} < \infty.
\end{equation}
Define 
\begin{equation} \label{taus}
\begin{aligned}
\tau'_n := \inf \{ t\leq \tau': \sup_x|u(t,x)| \geq n \}, \\
\tau''_n := \inf \{ t\leq \tau': \sup_x |v(t,x)| \geq n \},
\end{aligned}
\end{equation}
and $\tau'''_n:=\tau'_n \wedge\tau''_n$. By \eqref{eqn 12.9.10},  $\tau'_n, \tau''_n$ and $\tau'''_n$ are stopping times.
Note also   that, due to \eqref{eqn 12.9.10}, $\tau'''_n \uparrow\tau'$ (a.s.).  Let us denote
\begin{equation} \label{xi-bar}
\bar{\xi} = \frac{\xi(t,x) |u|^{1+\lambda}}{u}\quad \left(\frac{0}{0}:=0\right).
\end{equation}
Then, since $|u| \leq n$ for $t\leq \tau'''_n$, we have $\sup_{\omega} \sup_{t\leq \tau'''_n}\|\bar{\xi}\|_{\infty} <\infty$.  Note that  $u$ satisfies
\begin{equation} \label{during proof 3}
du =(au_{xx}+bu+cu)dt +\sum_{k=1}^{\infty} \bar{\xi}u\eta_k dw_t^k,\quad t\leq \tau'''_n,
\end{equation}
and therefore by Maximum principle, Theorem \ref{thm maximum},  we conclude $u\geq 0$ for $t\leq \tau'''_n$. Similarly, $v\geq 0$ for  $t\leq \tau'''_n$. Thus $|(-n) \vee u \wedge n|=u$ and $|(-n) \vee v\wedge n|=v$ for $t\leq \tau'''_n$.  Hence, by the uniqueness of Lemma \ref{cut_off_lemma}, we conclude $u=v$ in $\mathfrak{H}_{p,\theta}^{1/2-\kappa}(I,\tau_m)$ for each $m$. This,
\eqref{eqn 12.9.11}, and the dominated convergence theorem imply   $u=v$ for $t\leq \tau'$.

{\bf{Step 2}}. Existence. Since $\nu:=1/2+\kappa+1/p-\theta/p \geq 0$,
$$
c_0:=\sup_x \psi^{\nu} <\infty.
$$
Take  $u_m \geq 0$ from Lemma \ref{cut_off_lemma} which belongs to  $\mathfrak{H}_{p,\theta}^{1/2-\kappa}(I,T)$ for any $T<\infty$ and satisfies
$$
du=(au_{xx}+bu_x+cu)dt+\sum_{k=1}^{\infty} \xi |u \wedge m|^{1+\lambda} \eta_k dw^k_t,\, 0<t<T\, ; \quad u(0)=u_0.
$$
 For positive integers $m, R$, define
\begin{equation*}
\tau_m^R := \inf\{ t\geq0:\sup_x|\psi^{-\frac{1}{2}-\kappa-\frac{1}{p}+\frac{\theta}{p}}u_m|\geq \frac{R}{c_0}\}.
\end{equation*}
Then, for $t\leq \tau_m^R$,
\begin{equation*}
\sup_{x}|u_m(t,x)|\leq\sup_{x}|\psi^{\frac{1}{2}+\kappa+\frac{1}{p}-\frac{\theta}{p}}(x)|\sup_{x}|\psi^{-\frac{1}{2}-\kappa-\frac{1}{p}+\frac{\theta}{p}}(x)u_m(t,x)|\leq R,
\end{equation*}
and therefore $u_m \wedge m=u_m \wedge m \wedge R$ for $t\leq \tau^R_m$. 
Thus, if $m\geq R$ then   both $u_m$ and $u_R$ satisfy 
\begin{equation*}
du =(au_{xx}+bu_x+cu)dt +\sum_{k=1}^{\infty} \xi |u\wedge R|^{1+\lambda}\eta_k dw_t^k,\quad t\leq \tau_m^R,\quad u(0,\cdot) = u_0.
\end{equation*}
Also, the same is true for
\begin{equation*}
\begin{aligned}
du =(au_{xx}+bu_x+cu)dt +\sum_{k=1}^{\infty} \xi |u\wedge m|^{1+\lambda}\eta_k dw_t^k,\quad t\leq \tau_R^R,\quad u(0,\cdot) = u_0.
\end{aligned}
\end{equation*}
Thus, by the uniqueness of  Lemma \ref{cut_off_lemma},   $u_m=u_R$ for all $t\leq  (\tau_m^R \vee \tau_R^R) \wedge T$, for any  positive integer $T$.  It follows that
$$
u_m=u_R, \quad \forall  t< \tau_m^R \vee \tau_R^R\quad (a.s.).
$$
  We conclude $\tau_R^R=\tau_m^R\leq \tau_m^m$ (a.s.).  Indeed, let $t<\tau_m^R$, then  
  $$\sup_{s\leq t}  \sup_x|\psi^{-\frac{1}{2}-\kappa-\frac{1}{p}+\frac{\theta}{p}}u_m(s)| 
  < \frac{R}{c_0}.
  $$ 
Since the left side above equals 
  $\sup_{s\leq t}  \sup_x|\psi^{-\frac{1}{2}-\kappa-\frac{1}{p}+\frac{\theta}{p}}u_R(s)|$, we get $t\leq \tau^R_R$, and thus $\tau^R_m \leq \tau^R_R$. A similar argument shows $\tau^R_m \geq \tau^R_R$.

  Thus we can define the function
$$ u(t,x) = u_m(t,x)\quad  \text{for} \quad t\in [0,\tau^m_m).
$$
Observe that $|u|\leq m$ for $t\leq \tau_m^m$ and $u$ satisfies \eqref{main_equation_1} for $t<\lim_{m\to\infty}\tau_m^m$.  Also, by Lemma \ref{key_lemma},
\begin{equation*}
\begin{aligned}
\limsup_{m\to\infty}P(\tau_m^m\leq T) = \limsup_{m\to\infty}P(\sup_{t\leq T,x}|\psi^{-\frac{1}{2}-\kappa-\frac{1}{p}+\frac{\theta}{p}}u_m(t,x)|\geq \frac{m}{c_0})\\
\leq \limsup_{m\to\infty}\sup_nP(\sup_{t\leq T,x}|\psi^{-\frac{1}{2}-1-\frac{1}{p}+\frac{\theta}{p}}u_n(t,x)|\geq \frac{m}{c_0})=0.
\end{aligned}
\end{equation*}
Thus $\tau^m_m \to \infty$ in probability, and since $\tau_m^m$ is increasing, $\tau^m_m \uparrow \infty$ (a.s).  It follows that $u$ satisfies equation \eqref{main_equation_1} for all $t<\infty$. 

Define $\tau_m=\tau^m_m \wedge m$. Then, since $u=u_m$ for $t\leq \tau_m$ and $u_m\in \mathfrak{H}_{p,\theta}^{1/2-\kappa}(I,T)$ for any $T<\infty$,  it follows that  $u\in  \mathfrak{H}_{p,\theta}^{1/2-\kappa}(I,\tau_m)$ for any $m$, and consequently $u\in  \mathfrak{H}_{p,\theta,loc}^{1/2-\kappa}(I,\infty)$.
The theorem is proved.

\bibliographystyle{plain}

\end{document}